\documentclass{article}
\usepackage[utf8]{inputenc}
\usepackage{amssymb,amstext,amsmath,amscd,amsthm,amsfonts,enumerate,latexsym}
\usepackage{color}
\usepackage{stmaryrd}

\usepackage{titlesec}
\titleformat{\section}{\centering\large\sc}{\thesection.}{0.7em}{}
\usepackage{hyperref}
\hypersetup{
    colorlinks=true,
    linkcolor=blue,
    citecolor=blue,
    filecolor=blue,      
    urlcolor=blue,
} 

\usepackage[title]{appendix}

\titleclass{\subsubsubsection}{straight}[\subsection]

\newcounter{subsubsubsection}[subsubsection]
\renewcommand\thesubsubsubsection{\thesubsubsection.\arabic{subsubsubsection}}

\titleformat{\subsubsubsection}
  {\normalfont\normalsize\bfseries}{\thesubsubsubsection}{1em}{}
\titlespacing*{\subsubsubsection}
{0pt}{3.25ex plus 1ex minus .2ex}{1.5ex plus .2ex}

\makeatletter
\renewcommand\paragraph{\@startsection{paragraph}{5}{\z@}%
  {3.25ex \@plus1ex \@minus.2ex}%
  {-1em}%
  {\normalfont\normalsize\bfseries}}
\renewcommand\subparagraph{\@startsection{subparagraph}{6}{\parindent}%
  {3.25ex \@plus1ex \@minus .2ex}%
  {-1em}%
  {\normalfont\normalsize\bfseries}}
\def\toclevel@subsubsubsection{4}
\def\toclevel@paragraph{5}
\def\toclevel@paragraph{6}
\def\l@subsubsubsection{\@dottedtocline{4}{7em}{4em}}
\def\l@paragraph{\@dottedtocline{5}{10em}{5em}}
\def\l@subparagraph{\@dottedtocline{6}{14em}{6em}}
\makeatother

\usepackage{bookmark}
\usepackage[backend=biber,style=alphabetic,sorting=nyt,maxbibnames=99,maxalphanames=99]{biblatex}

\newcommand{\sortas}[1]{}
    \usepackage{afterpage}
\usepackage[all]{xy}
\usepackage{tikz-cd}

\begin{document}
\theoremstyle{plain}
\newtheorem{thm}{Theorem}[section]
\newtheorem{theorem}[thm]{Theorem}

\numberwithin{equation}{thm}
\newtheorem*{thm*}{Theorem}
\newtheorem*{cor*}{Corollary}
\newtheorem*{thma}{Theorem A}
\newtheorem*{thmb}{Theorem B}
\newtheorem*{mthm}{Main Theorem}
\newtheorem*{mcor}{Theorem \ref{maincor}}
\newtheorem{thms}[thm]{Theorems}
\newtheorem{prop}[thm]{Proposition}
\newtheorem{proposition}[thm]{Proposition}
\newtheorem{prodef}[thm]{Proposition and Definition}
\newtheorem{lemma}[thm]{Lemma}
\newtheorem{lem}[thm]{Lemma}
\newtheorem{cor}[thm]{Corollary}

\newtheorem{corollary}[thm]{Corollary}
\newtheorem{claim}[thm]{Claim}
\newtheorem*{claim*}{Claim}

\theoremstyle{definition}
\newtheorem{defn}[thm]{Definition}
\newtheorem{setup}[thm]{Setup}
\newtheorem{defnProp}[thm]{Definition and Proposition}
\newtheorem{prob}[thm]{Problem}
\newtheorem{question}[thm]{Question}
\newtheorem{definition}[thm]{Definition}
\newtheorem{ex}[thm]{Example}
\newtheorem{exercise}[thm]{Exercise}
\newtheorem{Example}[thm]{Example}
\newtheorem*{notation}{Notation}
\newtheorem*{acknowledgements}{Acknowledgements}
\newtheorem{example}[thm]{Example}
\newtheorem{fact}[thm]{Fact}
\newtheorem*{Fact*}{Fact}
\newtheorem{conj}[thm]{Conjecture}
\newtheorem{ques}[thm]{Question}
\newtheorem*{quesa}{Question A}
\newtheorem*{quesb}{Question B}
\newtheorem{case}{Case}
\newtheorem{setting}[thm]{Setting}
\newtheorem{rem}[thm]{Remark}
\newtheorem{note}[thm]{Note}
\newtheorem{remark}[thm]{Remark}

\theoremstyle{remark}
\newtheorem*{lemm}{Lemma}
\newtheorem*{pf}{{\sl Proof}}
\newtheorem*{pfpr}{Proof of Proposition \eqref{positive Hecke}}
\newtheorem*{tpf}{{\sl Proof of Theorem 1.1}}
\newtheorem*{cpf1}{{\sl Proof of Claim 1}}
\newtheorem*{cpf2}{{\sl Proof of Claim 2}}

\def\soc{\operatorname{Soc}}
\def\xx{\text{{\boldmath$x$}}}
\def\L{\mathrm{U}}
\def\Z{\mathcal{Z}}
\def\D{\mathcal{D}}
\def\Min{\operatorname{Min}}
\def\Gdim{\operatorname{Gdim}}
\def\pd{\operatorname{pd}}
\def\GCD{\operatorname{GCD}}
\def\Ext{\operatorname{Ext}}
\def\X{\mathcal{X}}
\def\S{\mathcal{S}}
\def\XX{\mathbb{X}}
\def\Im{\operatorname{Im}}
\def\s{\operatorname{s}}
\def\Ker{\operatorname{Ker}}
\def\KK{\mathbb{K}}
\def\bbZ{\mathbb{Z}}
\def\LL{\mathbb{U}}
\def\E{\operatorname{E}}
\def\H{\operatorname{H}}

\def\Hom{\operatorname{Hom}}
\def\Ext{\operatorname{Ext}}
\def\RHom{\mathrm{{\bf R}Hom}}
\def\Tor{\operatorname{Tor}}
\def\Max{\operatorname{Max}}
\def\Assh{\operatorname{Assh}}
\def\End{\mathrm{End}}
\def\rad{\mathrm{rad}}
\def\Soc{\mathrm{Soc}}
\def\Proj{\operatorname{Proj}}

\def\Mod{\mathrm{Mod}}
\def\mod{\mathrm{mod}}
\def\G{{\sf G}}
\def\T{{\sf T}}

\def\Coker{\mathrm{Coker}}
\def\im{\mathrm{im}\;}
\def\st{{}^{\ast}}
\def\rank{\mathrm{rank}}
\def\a{\mathfrak a}
\def\b{\mathfrak b}
\def\c{\mathfrak c}
\def\e{\mathrm{e}}
\def\f{\mathfrak{f}}
\def\m{\mathfrak m}
\def\n{\mathfrak n}
\def\p{\mathfrak p}
\def\q{\mathfrak q}
\def\P{\mathfrak P}
\def\Q{\mathfrak Q}
\def\N{\Bbb N}
\def\C{\mathcal{C}}
\def\K{\mathrm{K}}
\def\H{\mathrm{H}}
\def\J{\mathrm{J}}
\def\Gr{\mathrm G}
\def\FC{\mathrm F}
\def\Var{\mathrm V}
\def\V{\mathrm V}
\def\r{\mathrm{r}}

\newcommand{\nCM}{\mathrm{nCM}}
\newcommand{\nSCM}{\mathrm{nSCM}}
\newcommand{\Aut}{\mathrm{Aut}}
\newcommand{\Att}{\mathrm{Att}}
\newcommand{\Ann}{\mathrm{Ann}}

\newcommand{\rma}{\mathrm{a}}
\newcommand{\rmb}{\mathrm{b}}
\newcommand{\rmc}{\mathrm{c}}
\newcommand{\rmd}{\mathrm{d}}
\newcommand{\rme}{\mathrm{e}}
\newcommand{\rmf}{\mathrm{f}}
\newcommand{\rmg}{\mathrm{g}}
\newcommand{\rmh}{\mathrm{h}}
\newcommand{\rmi}{\mathrm{i}}
\newcommand{\rmj}{\mathrm{j}}
\newcommand{\rmk}{\mathrm{k}}
\newcommand{\rml}{\mathrm{l}}
\newcommand{\rmm}{\mathrm{m}}
\newcommand{\rmn}{\mathrm{n}}
\newcommand{\rmo}{\mathrm{o}}
\newcommand{\rmp}{\mathrm{p}}
\newcommand{\rmq}{\mathrm{q}}
\newcommand{\rmr}{\mathrm{r}}
\newcommand{\rms}{\mathrm{s}}
\newcommand{\rmt}{\mathrm{t}}
\newcommand{\rmu}{\mathrm{u}}
\newcommand{\rmv}{\mathrm{v}}
\newcommand{\rmw}{\mathrm{w}}
\newcommand{\rmx}{\mathrm{x}}
\newcommand{\rmy}{\mathrm{y}}
\newcommand{\rmz}{\mathrm{z}}

\newcommand{\rmA}{\mathrm{A}}
\newcommand{\rmB}{\mathrm{B}}
\newcommand{\rmC}{\mathrm{C}}
\newcommand{\rmD}{\mathrm{D}}
\newcommand{\rmE}{\mathrm{E}}
\newcommand{\rmF}{\mathrm{F}}
\newcommand{\rmG}{\mathrm{G}}
\newcommand{\rmH}{\mathrm{H}}
\newcommand{\rmI}{\mathrm{I}}
\newcommand{\rmJ}{\mathrm{J}}
\newcommand{\rmK}{\mathrm{K}}
\newcommand{\rmL}{\mathrm{L}}
\newcommand{\rmM}{\mathrm{M}}
\newcommand{\rmN}{\mathrm{N}}
\newcommand{\rmO}{\mathrm{O}}
\newcommand{\rmP}{\mathrm{P}}
\newcommand{\rmQ}{\mathrm{Q}}
\newcommand{\rmR}{\mathrm{R}}
\newcommand{\rmS}{\mathrm{S}}
\newcommand{\rmT}{\mathrm{T}}
\newcommand{\rmU}{\mathrm{U}}
\newcommand{\rmV}{\mathrm{V}}
\newcommand{\rmW}{\mathrm{W}}
\newcommand{\rmX}{\mathrm{X}}
\newcommand{\rmY}{\mathrm{Y}}
\newcommand{\rmZ}{\mathrm{Z}}

\newcommand{\cala}{\mathcal{a}}
\newcommand{\calb}{\mathcal{b}}
\newcommand{\calc}{\mathcal{c}}
\newcommand{\cald}{\mathcal{d}}
\newcommand{\cale}{\mathcal{e}}
\newcommand{\calf}{\mathcal{f}}
\newcommand{\calg}{\mathcal{g}}
\newcommand{\calh}{\mathcal{h}}
\newcommand{\cali}{\mathcal{i}}
\newcommand{\calj}{\mathcal{j}}
\newcommand{\calk}{\mathcal{k}}
\newcommand{\call}{\mathcal{l}}
\newcommand{\calm}{\mathcal{m}}
\newcommand{\caln}{\mathcal{n}}
\newcommand{\calo}{\mathcal{o}}
\newcommand{\calp}{\mathcal{p}}
\newcommand{\calq}{\mathcal{q}}
\newcommand{\calr}{\mathcal{r}}
\newcommand{\cals}{\mathcal{s}}
\newcommand{\calt}{\mathcal{t}}
\newcommand{\calu}{\mathcal{u}}
\newcommand{\calv}{\mathcal{v}}
\newcommand{\calw}{\mathcal{w}}
\newcommand{\calx}{\mathcal{x}}
\newcommand{\caly}{\mathcal{y}}
\newcommand{\calz}{\mathcal{z}}

\newcommand{\calA}{\mathcal{A}}
\newcommand{\calB}{\mathcal{B}}
\newcommand{\calC}{\mathcal{C}}
\newcommand{\calD}{\mathcal{D}}
\newcommand{\calE}{\mathcal{E}}
\newcommand{\calF}{\mathcal{F}}
\newcommand{\calG}{\mathcal{G}}
\newcommand{\calH}{\mathcal{H}}
\newcommand{\calI}{\mathcal{I}}
\newcommand{\calJ}{\mathcal{J}}
\newcommand{\calK}{\mathcal{K}}
\newcommand{\calL}{\mathcal{L}}
\newcommand{\calM}{\mathcal{M}}
\newcommand{\calN}{\mathcal{N}}
\newcommand{\calO}{\mathcal{O}}
\newcommand{\calP}{\mathcal{P}}
\newcommand{\calQ}{\mathcal{Q}}
\newcommand{\calR}{\mathcal{R}}
\newcommand{\calS}{\mathcal{S}}
\newcommand{\calT}{\mathcal{T}}
\newcommand{\calU}{\mathcal{U}}
\newcommand{\calV}{\mathcal{V}}
\newcommand{\calW}{\mathcal{W}}
\newcommand{\calX}{\mathcal{X}}
\newcommand{\calY}{\mathcal{Y}}
\newcommand{\calZ}{\mathcal{Z}}

\newcommand{\fka}{\mathfrak{a}}
\newcommand{\fkb}{\mathfrak{b}}
\newcommand{\fkc}{\mathfrak{c}}
\newcommand{\fkd}{\mathfrak{d}}
\newcommand{\fke}{\mathfrak{e}}
\newcommand{\fkf}{\mathfrak{f}}
\newcommand{\fkg}{\mathfrak{g}}
\newcommand{\fkh}{\mathfrak{h}}
\newcommand{\fki}{\mathfrak{i}}
\newcommand{\fkj}{\mathfrak{j}}
\newcommand{\fkk}{\mathfrak{k}}
\newcommand{\fkl}{\mathfrak{l}}
\newcommand{\fkm}{\mathfrak{m}}
\newcommand{\fkn}{\mathfrak{n}}
\newcommand{\fko}{\mathfrak{o}}
\newcommand{\fkp}{\mathfrak{p}}
\newcommand{\fkq}{\mathfrak{q}}
\newcommand{\fkr}{\mathfrak{r}}
\newcommand{\fks}{\mathfrak{s}}
\newcommand{\fkt}{\mathfrak{t}}
\newcommand{\fku}{\mathfrak{u}}
\newcommand{\fkv}{\mathfrak{v}}
\newcommand{\fkw}{\mathfrak{w}}
\newcommand{\fkx}{\mathfrak{x}}
\newcommand{\fky}{\mathfrak{y}}
\newcommand{\fkz}{\mathfrak{z}}

\newcommand{\fkA}{\mathfrak{A}}
\newcommand{\fkB}{\mathfrak{B}}
\newcommand{\fkC}{\mathfrak{C}}
\newcommand{\fkD}{\mathfrak{D}}
\newcommand{\fkE}{\mathfrak{E}}
\newcommand{\fkF}{\mathfrak{F}}
\newcommand{\fkG}{\mathfrak{G}}
\newcommand{\fkH}{\mathfrak{H}}
\newcommand{\fkI}{\mathfrak{I}}
\newcommand{\fkJ}{\mathfrak{J}}
\newcommand{\fkK}{\mathfrak{K}}
\newcommand{\fkL}{\mathfrak{L}}
\newcommand{\fkM}{\mathfrak{M}}
\newcommand{\fkN}{\mathfrak{N}}
\newcommand{\fkO}{\mathfrak{O}}
\newcommand{\fkP}{\mathfrak{P}}
\newcommand{\fkQ}{\mathfrak{Q}}
\newcommand{\fkR}{\mathfrak{R}}
\newcommand{\fkS}{\mathfrak{S}}
\newcommand{\fkT}{\mathfrak{T}}
\newcommand{\fkU}{\mathfrak{U}}
\newcommand{\fkV}{\mathfrak{V}}
\newcommand{\fkW}{\mathfrak{W}}
\newcommand{\fkX}{\mathfrak{X}}
\newcommand{\fkY}{\mathfrak{Y}}
\newcommand{\fkZ}{\mathfrak{Z}}
\newcommand{\kr}{\mathrm{K}_R}
\newcommand{\mapright}[1]{%
\smash{\mathop{%
\hbox to 1cm{\rightarrowfill}}\limits^{#1}}}

\newcommand{\mapleft}[1]{%
\smash{\mathop{%
\hbox to 1cm{\leftarrowfill}}\limits_{#1}}}

\newcommand{\mapdown}[1]{\Big\downarrow
\llap{$\vcenter{\hbox{$\scriptstyle#1\,$}}$ }}
\renewcommand{\o}[1]{{#1}^{\circ}}

\newcommand{\mapup}[1]{\Big\uparrow
\rlap{$\vcenter{\hbox{$\scriptstyle#1\,$}}$ }}
\def\grade{\operatorname{grade}}
\def\depth{\operatorname{depth}}
\def\AGL{\operatorname{AGL}}
\def\Supp{\operatorname{Supp}}
\def\ann{\operatorname{Ann}}
\def\Ass{\operatorname{Ass}}
\def\height{\mathrm{ht}}
\def\Sp{\operatorname{Spec}}
\def\Spf{\operatorname{Spf}}
\def\Syz{\mathrm{Syz}}
\def\hdeg{\operatorname{hdeg}}
\def\id{\operatorname{id}}
\def\gr{\mbox{\rm gr}}
\def\Zdv{\operatorname{Zdv}}
\def\Wd{\operatorname{\mathrm{WD}}}
\def\r{\mathrm{r}}

\def\A{{\mathcal A}}
\def\B{{\mathcal B}}
\def\F{{\mathcal F}}
\def\Y{{\mathcal Y}}
\def\W{{\mathcal W}}
\def\M{{\mathcal M}}
\def\O{{\mathcal O}}
\def\H{{\mathcal H}}
\def\G{{\mathcal G}}
\def\R{{\mathcalR}}
\def\I{{\mathcal I}}
\def\J{{\mathcal J}}
\def\L{{\mathcal L}}
\def\U{{\mathcal U}}
\def\fM{\mathfrak M}
\def\fN{\mathfrak N}
\def\fl{\pi^\flat}
\def\E{{\mathcal E}}
\renewcommand{\R}{\mathcal{R}}

\def\yy{\text{\boldmath $y$}}
\def\rb{\overline{R}}
\def\ol{\overline}
\def\Deg{\mathrm{deg}}

\def\PP{\mathbb{P}}
\def\II{\mathbb{I}}
\newcommand{\w}{\wedge}
\newcommand{\dep}{\mathrm{depth}}
\newcommand{\kz}[1]{\mathrm{K}_{\bullet}(#1)}
\newcommand{\tc}[2]{\langle #1, #2 \rangle}
\newcommand{\hm}[2]{\operatorname{Hom}_{R}(#1, #2 )}
\newcommand{\pt}{\partial}
\newcommand{\ckz}[1]{\mathrm{K}_{\bullet}(#1)}
\newcommand{\loc}[1]{\mathrm{H}_{\m}^d({#1})}
\newcommand{\hy}[2]{\mathrm{H}_{\m}^{#1}(#2)}
\newcommand{\cy}[2]{\mathrm{H}_{#1}(#2)}
\newcommand{\lc}[3]{\mathrm{H}_{#1}^{#2}(#3)}
\newcommand{\exe}[2]{\operatorname{Ext}_{R}^{#1}(#2,\kr)}
\newcommand{\ext}[3]{\operatorname{Ext}_{R}^{#1}(#2, #3)}
\newcommand{\ie}[1]{\mathrm{E}_{R}(#1)}
\newcommand{\vin}{\rotatebox[origin=c]{90}{$\in$}}
\newcommand{\ts}{\otimes}
\newcommand{\as}[1]{\operatorname{Ass}_{R} #1}
\newcommand{\mcF}{\mathcal{F}}
\newcommand{\idd}[1]{\mathrm{id}_{R}({#1})}
\newcommand{\wrr}{\widehat{R}}
\newcommand{\mi}[1]{\mu^i(\m,{#1})}
\newcommand{\krr}[1]{\mathrm{K}_{#1}}
\newcommand{\rt}[1]{\mathrm{r}_R({#1})}
\newcommand{\ci}{\mathbb{C}_{i}(R)}
\newcommand{\ccr}[1]{\mathbb{C}_{#1}(R)}
\newcommand{\dg}{\textsuperscript{\textdagger}}
\newcommand{\rr}[2]{\mathrm{r}_{#1}(#2)}
\newcommand{\wm}{\widehat{\m}}
\newcommand{\wE}{\E_{\wrr}(\wrr/\wm)}
\renewcommand{\sp}{\mathrm{Spec}_S(\mathcal{A})}
\newcommand{\proj}{\mathrm{Proj}}

\newcommand{\hms}[2]{\mathrm{Hom}_S(#1,#2)}
\newcommand{\hyy}[3]{\mathrm{H}_{\underline{#1}}^{#2}(#3)}
\newcommand{\rmod}{R-\mathrm{Mod}}
\newcommand{\tth}{\text{th}}
\newcommand{\asl}[1]{\operatorname{Ass}_{S^{-1}R}S^{-1} #1}
\newcommand{\uijk}{U_{ij\lambda}}
\renewcommand{\Proj}{\mathrm{Proj}}
\renewcommand{\rad}{\mathrm{rad}}
\newcommand{\norm}[1]{\|#1\|}
\newcommand{\fil}[1]{\mathrm{Fil}^{#1}}
\newcommand{\Ql}{\overline{\mathbf{Q}}_{\ell}}
\newcommand{\GL}{\mathrm{GL}_n(\overline{\mathbf{Q}}_{\ell})}
\newcommand{\un}{F^{\mathrm{unr}}}
\newcommand{\tr}{F^{\mathrm{tr}}}
\newcommand{\tl}{t_{\ell}}
\newcommand{\gl}{\mathrm{GL}}
\newcommand{\en}{\mathrm{End}}
\newcommand{\spe}{\mathrm{SL}_2(\mathbb{F}_p)}
\newcommand{\gen}{\mathrm{GL}_2(\mathbb{F}_p)}
\newcommand{\glr}{\mathrm{GL}_n(\mathbb{R})}
\newcommand{\supp}{\mathrm{supp}\;}
\newcommand{\Qbar}{\overline{\mathbf{Q}}}
\newcommand{\Zbar}{\overline{\mathbb{Z}}_p}
\newcommand{\oks}{\O_{K,S}}
\newcommand{\cl}{\mathrm{Cl}}
\newcommand{\cls}{\mathrm{cl}}
\newcommand{\ord}{\mathrm{ord}}
\newcommand{\ab}{\mathrm{ab}}
\newcommand{\unr}{\mathrm{unr}}
\newcommand{\gal}{\mathrm{Gal}}
\newcommand{\vol}{\mathrm{vol}}
\newcommand{\reg}{\mathrm{Reg}}
\newcommand{\re}{\mathrm{Re}}
\newcommand{\spm}{\mathrm{Spm}}
\newcommand{\kunr}{K^\mathrm{unr}}
\newcommand{\kt}{K^\mathrm{t}}
\newcommand{\pnm}{\mathbb{P}^n_R\times_R\mathbb{P}^m_R}
\newcommand{\pnmr}{\mathbb{P}_R^{nm+n+m}}
\newcommand{\bk}{\mathrm{Breuil}\textendash\mathrm{Kisin}}
\newcommand{\bkf}{\mathrm{Breuil}\textendash\mathrm{Kisin}\textendash\mathrm{Fargues}}
\renewcommand{\ss}{\mathrm{ss}}
\newcommand{\crys}{\mathrm{crys}}
\newcommand{\Ao}{A^{\circ}}
\newcommand{\Ddr}{D_{\mathrm{dR}}(\fM_{\Ao}^{\inf})}
\newcommand{\und}[1]{\underline{#1}}
\newcommand{\xdss}{\X_d^{\ss,\underline{\lambda},\tau}}
\newcommand{\xdcrys}{\X_d^{\crys,\underline{\lambda},\tau}}
\newcommand{\BK}{\mathrm{BK}}
\newcommand{\disc}{\mathrm{disc}}
\newcommand{\kbas}{K^\mathrm{basic}}
\newcommand{\cyc}{\mathrm{cyc}}
\newcommand{\aka}{\mathbf{A}_{K,A}}

\newcommand{\ak}[1]{\mathbf{A}_{K,#1}}
\newcommand{\lt}{\mathrm{LT}}
\newcommand{\xan}{X^{\mathrm{an}}}
\newcommand{\ur}[1]{\mathrm{ur}_{#1}}

\newenvironment{psmallmatrix}
  {\left(\begin{smallmatrix}}
  {\end{smallmatrix}\right)}
  \newcommand{\eg}{\mathrm{EG}}
\newcommand{\ltet}{\M^{\mathrm{\acute{e}t},\lt}_{\varphi,\Gamma}}
\renewcommand{\ext}{\mathrm{Ext}}

\title{\large\bf Moduli stacks of Lubin--Tate $(\varphi,\Gamma)$-modules}
\author{{\sc\normalsize Dat Pham}}
\date{}
\maketitle
\begin{abstract}
    We define and study stacks which parametrize Lubin--Tate $(\varphi,\Gamma)$-modules. By working at a perfectoid level, we compare these with the moduli stacks of cyclotomic $(\varphi,\Gamma)$-modules in \cite{EG22}. As a consequence, we deduce perfectness of the Herr complex in the Lubin--Tate setting. 
\end{abstract}
\section{Introduction}
Let $K/\mathbf{Q}_p$ be a finite extension with algebraic closure $\overline{K}$. A basic idea in $p$-adic Hodge theory is to consider ``deeply ramified'' subextensions $K\subseteq K_\infty \subseteq \overline{K}$ which encodes most interesting ramifications of $\overline{K}/K$, yet whose Galois group is simply enough to control. When $K_\infty/K$ is the cyclotomic extension, this idea was realized in \cite{Fon90}, where Fontaine introduced the notion of \' etale $(\varphi,\Gamma)$-modules, and showed that they are naturally equivalent to continuous representations of $G_K$ on finite $\mathbf{Z}_p$-modules. Since its introduction, the concept of $(\varphi,\Gamma)$-modules has proved to be a very powerful tool in the study of $p$-adic Galois representations. It has featured prominently in e.g. Berger's proof of Fontaine's $p$-adic monodromy conjecture, and in Colmez's celebrated work on the $p$-adic local Langlands correspondence for $\mathrm{GL}_2(\mathbf{Q}_p)$.

In \cite{EG22}, Emerton and Gee define and study stacks which interpolate Fontaine's $(\varphi,\Gamma)$-modules in families. More precisely, for each integer $d\geq 1$, they consider the stack $\X_d$ over $\Spf\mathbf{Z}_p$ whose groupoid of $A$-valued points, for any $p$-adically complete $\mathbf{Z}_p$-algebra $A$, is given by the groupoid of rank $d$ projective \' etale $(\varphi,\Gamma)$-modules with $A$-coefficients. The geometry of $\X_d$ has been studied extensively in \cite{EG22}. In particular, the authors show that $\X_d$ is a Noetherian formal algebraic stack, and moreover, its underlying reduced substack is an algebraic stack of finite type over $\mathbf{F}_p$, whose irreducible components admit a natural labelling by Serre weights. The stack $\X_d$ is also expected to play a critical role in the emerging categorical $p$-adic local Langlands program (cf. \cite{EGH}). 

With an eye toward realizing a $p$-adic local Langlands correspondence for fields other than $\mathbf{Q}_p$, there has been a growing interest in studying the analogue of Fontaine's notion in which the cyclotomic extension $K_\infty/K$ is replaced by a Lubin--Tate extension (in what follows, we will often refer to these objects simply as Lubin--Tate $(\varphi,\Gamma)$-modules). We will not try to survey these results, but instead refer the reader, for instance, to \cite{KisinRen}, \cite{Schneider}, and \cite{HerrLTsetting}.

The goal of this note is to extend the aforementioned construction of Emerton--Gee to the Lubin--Tate setting. Before stating our main results, we need to introduce some notation. Let $\pi \in K$ be a fixed uniformizer, and let $\F_{\phi}$ be a Lubin-Tate group for $\pi$, with Frobenius power
series $\phi(T)\in \O_K[[T]]$. The corresponding ring map $\O_K\to\mathrm{End}(\F_{\phi})$ is denoted by $a\mapsto [a](T)$; in particular $[\pi](T) =\phi(T)$. Let $K_\infty/K$ be the extension generated by the torsion points of $\F_{\phi}$ and let $\chi: \Gamma:=\mathrm{Gal}(K_\infty/K)\xrightarrow{\sim}\O_K^\times$ be the resuting Lubin--Tate character. For the purpose of this introduction, we simply let
\begin{displaymath}
    \mathbf{A}_K:=\widehat{\O_K((T))}=\left\{\sum_{n\in\mathbf{Z}}a_nT^n\;|\;\text{$a_n\in \O_K$ and $a_n\to 0$ as $n\to -\infty$}\right\},
\end{displaymath}
where the hat denotes $\pi$-adic completion. The ring $\mathbf{A}_K$ is further endowed with a Frobenius $\varphi_q: f(T)\mapsto f([\pi](T))$ and an action of $\Gamma$ given by $(g,f(T))\mapsto f([\chi(g)](T))$ for $g\in {\Gamma}$; as the notation suggests, $\varphi_q$ is a lift of the $q$-power Frobenius modulo $\pi$. An \' etale $(\varphi_q,\Gamma)$-module (over $\mathbf{A}_K$) is then, by definition, a finite $\mathbf{A}_K$-module endowed
with commuting continuous semilinear actions of $\varphi_q$ and $\Gamma$ such that the linearization of $\varphi_q$ an isomorphism. There is again a natural equivalence between \' etale $(\varphi_q,\Gamma)$-modules and representations of $G_K$ on finite $\O_K$-modules. (One can be slightly more general by allowing also representations on $\O_F$-modules with $F$ being a finite subextension of $K/\mathbf{Q}_p$. See Section \eqref{definition related} below.)  

Fix now an integer $d\geq 1$. By definition, our stack $\X_{K,d}^{\lt}$ takes a $\pi$-adically complete $\O_K$-algebra $A$ to the groupoid of rank $d$ projective \' etale $(\varphi_q,\Gamma)$-modules over $\aka:=\mathbf{A}_K\widehat{\otimes}_{\O_K}A$.
\begin{thm}[Proposition \eqref{Xd ind algebraic K general}]\label{Xd ind algebraic intro}
$\X_{K,d}^{\lt}$ is a limit preserving Ind-algebraic stack over $\Spf \O_K$, with finitely presented affine diagonal.
\end{thm}
The proof of Theorem \eqref{Xd ind algebraic intro} follows closely the argument used in \cite{EG22} for the stack $\X_{K,d}^{\eg}$ of rank $d$ projective \textit{cyclotomic} \' etale $(\varphi,\Gamma)$-modules. Namely, we will deduce the claimed properties for $\X_{K,d}^{\lt}$ from the corresponding properties of the stack of \' etale $\varphi_q$-modules.

The next result gives a comparison between $\X_{K,d}^{\lt}$ and the stack $\X_{K,d}^{\eg}$ in \cite{EG22}.
\begin{thm}[Corollary \eqref{equivalence LT EG stacks}]\label{equivalence LT EG stacks intro}
    There is an isomorphism
\begin{align*}
    \X^{\lt}_{K,d}\xrightarrow{\sim} \X^{\eg}_{K,d}.
\end{align*}
\end{thm}
The proof proceeds by using the descent results in \cite{EG22} to reduce the statement to a comparison between \' etale $\varphi_p$-modules over $W(\mathbf{C}^\flat)\widehat{\otimes}_{\mathbf{Z}_p}A$ and \' etale $\varphi_q$-modules over $W_{\O_K}(\mathbf{C}^\flat)\widehat{\otimes}_{\O_K}A$.

As a consequence of Theorem \eqref{equivalence LT EG stacks intro} and the results in \cite{EG22}, we deduce the following refinement of Theorem \eqref{Xd ind algebraic intro} regarding the geometry of the stack of $\X_{K,d}^{\lt}$.
\begin{cor}[Corollary \eqref{Xd noetherian formal}]
$\X_{K,d}^{\lt}$ is a Noetherian formal algebraic stack over $\Spf \O_K$. The underlying reduced substack $\X_{d,\mathrm{red}}^{\lt}$ is an algebraic stack of finite presentation over $\mathbf{F}$. Moreover, the irreducible components of $\X_{d,\mathrm{red}}^{\lt}$ admits a natural labeling by Serre weights. 
\end{cor}
We also introduce a version of the Herr complex (\cite{Her98}) in the Lubin--Tate setting with coefficients, and give a new proof of the fact that this complex computes Galois cohomology (Theorem \eqref{comparison Galois coho}). We refer the reader to Subsection \eqref{herr complex section} for the definition of this complex. Finally, by using again the above comparison (Theorem \eqref{equivalence LT EG stacks intro}), we are able to deduce the following result, which may be of independent interest.
\begin{thm}[Theorem \eqref{Herr perfect}]
Let $A$ is a finite type $\pi$-nilpotent $\O_K$-algebra, and let $M$ be a finite projective \' etale $(\varphi_q,\Gamma)$-module with $A$-coefficients. Then the Lubin--Tate Herr complex associated to $M$ is a perfect complex of $A$-modules, whose formation commutes with arbitrary finite type base change in $A$.
\end{thm}
\begin{acknowledgements}
I am grateful to my advisors B. V. Le Hung and S. Morra for their constant encouragement and for helpful discussions. This project has received funding from the European Union’s Horizon 2020 research and innovation programme under the Marie Sk\l odowska-Curie grant agreement No. 945322.
\end{acknowledgements}
\section{Generalities on \texorpdfstring{$\varphi_q$}{phi}-modules}\label{definition related}
\subsection{Setup}\label{situation}
Fix a complete algebraic closure $\mathbf{C}:=\widehat{\overline{\mathbf{Q}}}_p$ of $\mathbf{Q}_p$. All algebraic extensions of $\mathbf{Q}_p$ will be regarded as subfields of $\mathbf{C}$. Let $F/\mathbf{Q}_p$ be a finite extension with ring of integers $\O_F$, uniformizer $\pi_F$, residue field $k$ of cardinality $q$. 

Let $\phi\in \O_F[[T]]$ be a Frobenius power series for $\pi$ (i.e. $\phi(T)\equiv \pi T\bmod{T^2}$ and $\phi(T) \equiv T^q\bmod{\pi}$) and let $\F_{\phi}$ be the corresponding Lubin-Tate formal group law over $\O_F$. As usual, denote by $[\cdot]_{\phi}: \O_F\to \mathrm{End}_{\O_F}(\F_{\phi})$ the unique ring homomorphism satisfying $[a]_{\phi}(T)\equiv aT\bmod{T^2}$ for all $a\in \O_F$, and $[\pi]_{\phi}=\phi$. We have a natural $\O_F$-module structure on $\m_{\overline{F}}$ given by $x+y:=\F_\phi(x,y)$ and $a.x:=[a]_{\phi}(x)$ for $x,y\in \m_{\overline{F}}$ and $a\in \O_F$. Let $F_n:=F(\F_{\phi}[\pi^n])$ be the field obtained by adjoining to $F$ all the $\pi^n$-torsion points of $\m_{\overline{F}}$, and set $F_{\infty}:=\bigcup_{n\geq 0}F_n$ (note that the extensions $F_n/F$ (and hence $F_{\infty}/F$) only depend on $\pi$ and not on the choice of $\phi$). 

We write $\chi: G_F\twoheadrightarrow \widetilde{\Gamma}_F\xrightarrow{\sim}\O_F^\times$ for the Lubin-Tate character associated to $\F$, where $\widetilde{\Gamma}_F:=\mathrm{Gal}(F_\infty/F)$\footnote{There is a slight conflict with the notation $\Gamma$ used in the Introduction; this is however to make it compatible with those in \cite{EG22}. We hope that this won't cause any confusion.}. We also let $T\F_{\phi}:=\varprojlim_n \F_{\phi}[\pi^n]$ be the Tate module of $\F$. This is a free $\O_F$-module of rank $1$. We have a $G_F$-equivariant map $\iota: T\F_{\phi}\to \O_{\widehat{F_\infty}}^\flat\subseteq\O_{\mathbf{C}}^\flat$ given by mapping $v=(v_0,v_1,\ldots)\in T\F_{\phi}$ to $\iota(v):=(v_0\bmod \pi,v_1\bmod \pi,\ldots)$. This map is indeed well-defined, since $\O_{\mathbf{C}}^\flat$ is naturally identified with $\varprojlim_{x\mapsto x^q}\O_{\mathbf{C}}/\pi$, and if $[\pi](v_{i+1})=v_i$, then $v_{i+1}^q\equiv v_i\bmod{\pi}$ by the defining property of $\phi$. Note also that the image of $\iota$ is in fact contained in the maximal ideal $\m_{\widehat{F_\infty}^\flat}$ (see \cite[Rem. 2.1.1]{Schneider}).

Finally, if $K$ is a finite extension of $F$, we let $K_n:=KF_n, K_{\infty}:=KF_{\infty}$, and $\widetilde{\Gamma}_K:=\mathrm{Gal}(K_{\infty}/K)$.
\subsection{Rings}\label{ring}
We now introduce the various coefficient rings for our Lubin--Tate $(\varphi_q,\widetilde{\Gamma}_K)$-modules. In what follows, $W_{\O_F}(A):=W(A)\otimes_{W(\mathbf{F}_q)}\O_F$ will denote the ring of ramified Witt vectors with coefficients in a perfect $\mathbf{F}_q$-algebra $A$. We also denote by $\varphi_q$ the natural Frobenius on $W_{\O_F}(A)$ (induced by functoriality and the $q$-power Frobenius on $A$).
\begin{lem}\label{lemma kisin ren}
There is a unique set-theoretic map $\{\cdot\}: \m_{\mathbf{C}^\flat}\to W_{\O_F}(\O_{\mathbf{C}}^\flat)$ lifting the inclusion $\m_{\mathbf{C}^\flat}\subseteq \O_{\mathbf{C}}^\flat$, such that $[\pi]_{\phi}(\{x\})=\varphi_q(\{x\})$ for all $x\in \m_{\mathbf{C}^\flat}$. Moreover, $\{\cdot\}$ respects the action of $G_F$ (where $G_F$ acts on $W_{\O_F}(\O_{\mathbf{C}}^\flat)$ by functoriality), and 
\begin{itemize}
\item[\emph{(i)}] For all $a\in \O_F$ and $v\in T\F_{\phi}$, the series $[a]_{\phi}(\{\iota(v)\})$ converges and we have $[a]_{\phi}(\{\iota(v)\})=\{\iota(av)\}$.
\item[\emph{(ii)}] The action of $G_F$ on $\{\iota(T\F_{\phi})\}\subseteq W_{\O_F}(\O_{\mathbf{C}}^\flat)$ factors through $\widetilde{\Gamma}_F$ and for $g\in \widetilde{\Gamma}_F$
\begin{displaymath}
[\chi(g)]_{\phi}(\{\iota(v)\})=\{\iota(g v)\}=\{g\cdot \iota(v)\}=g\cdot \{\iota(v)\}.
\end{displaymath}
\end{itemize}.
\end{lem}
\begin{proof}
This is \cite[Lem. 9.3]{Col02}. Concretely, the map $x\mapsto \{x\}$ is given by 
\begin{displaymath}
    \{x\}:=\lim_n [\pi^n]_{\phi}([x^{q^{-n}}])
\end{displaymath}
(recall that we always endow $W_{\O_K}(\O_{\mathbf{C}}^\flat)$ with the \textit{weak topology}, i.e. the $(\pi,u)$-adic topology, where $u\in \mathbf{C}^\flat$ is an arbitrary pseudo-uniformizer). (Strictly speaking, unlike \textit{loc. cit.}, here $\{\cdot\}$ is only defined on $\m_{\mathbf{C}^\flat}$. This is because we are not assuming that $\phi$ is actually a polynomial; consequently, restricting the domain to $\m_{\mathbf{C}^\flat}$ is necessary to ensure that the limit defining $\{x\}$ indeed exists. See \cite[Lem. 2.1.11]{Schneider} and the surrounding material.)
\end{proof}
We begin with the case where $K=F$. Let $v\in T\F_{\phi}$ be an $\O_F$-generator. There is an embedding $k[[T]]\hookrightarrow \O_{\widehat{F_\infty}}^\flat\subseteq \O_{\mathbf{C}}^\flat$ sending $T\mapsto \iota(v)$, whose image is identified with the ring of integers of the imperfect norm field $\mathbf{E}_F':=X_F(F_\infty)$ (cf. \cite[\textsection 2]{Win83}) via a canonical embedding $X_F(F_\infty)\hookrightarrow \widehat{F_\infty}^\flat\subseteq {\mathbf{C}}^\flat$ (this embedding identifies the target with the completion of the perfect closure of the source, see \cite[Cor. 4.3.4]{Win83}). Since $\{\iota(v)\}$ is a lift of $\iota(v)$ by construction, we obtain an embedding 
\begin{displaymath}
\O_F[[T]]\hookrightarrow W_{\O_F}(\O_{\widehat{F_\infty}}^\flat)\subseteq W_{\O_F}(\O_{\mathbf{C}}^\flat)
\end{displaymath}
sending $T\mapsto \{\iota(v)\}$ (the source being endowed with the $(\pi,T)$-topology, and the target with its weak topology), which extends further to a map
\begin{displaymath}
\widehat{\O_F((T))}\hookrightarrow W_{\O_F}(\widehat{F_\infty}^\flat)\subseteq W_{\O_F}(\mathbf{C}^\flat),
\end{displaymath}
whose image we denote by $\mathbf{A}_F'$ (where the hat denotes the $\pi$-adic completion). It is a complete discrete valuation ring with uniformizer $\pi$ and residue field $X_F(F_\infty)\cong k((T))$, which is in fact independent on the choice of $v\in T\F_{\phi}$ (\cite[Rem. 2.1.17]{Schneider}). By Lemma \eqref{lemma kisin ren}, $\mathbf{A}_F'$ are stable by $\varphi_q$ and by the natural action of $G_F$ on $W_{\O_F}(\O_{\mathbf{C}}^\flat)$. Moreover, the $G_F$-action on $\mathbf{A}_F'$ factors through $\widetilde{\Gamma}_F$. 

To summarize, $\mathbf{A}_F'$ is $(\varphi_q,\widetilde{\Gamma}_F)$-equivariantly isomorphic to the ring  
\begin{displaymath}
\widehat{\O_F((T))}\cong\left\{\sum_{n\in\mathbf{Z}}a_nT^n\;|\;\text{$a_n\in \O_F$ and $a_n\to 0$ as $n\to -\infty$}\right\},
\end{displaymath}
where the actions of $\varphi_q$ and $\widetilde{\Gamma}_F$ on the latter ring are given by $\varphi_q: f(T)\mapsto f([\pi]_{\phi}(T))$ and $g: f(T)\mapsto f([\chi(g)]_{\phi}(T))$ for $g\in \widetilde{\Gamma}_F$.

We now return to the case of a general finite extension $K/F$. Since $K_\infty=KF_\infty$ by definition, the norm field $\mathbf{E}_K':=X_K(K_\infty)$ associated to the extension $K_\infty/K$ is a finite separable extension of $\mathbf{E}_F'$ (despite the notation, the field $X_K(K_\infty)$ depends only on $K_\infty$ and not on $K$ itself, see \cite[Rem. 2.1.4]{Win83}). As $\mathbf{B}_{F}':=\mathbf{A}_F'[1/p]$ is a discrete valued field with uniformizer $\pi$ and residue field $\mathbf{E}_F'$, it follows that there is a unique finite unramified
extension of $\mathbf{B}_F'$ contained in the field $W(\mathbf{C}^\flat)[1/p]$ with residue field $\mathbf{E}_K'$. We
denote this extension by $\mathbf{B}_K'$, and by $\mathbf{A}_K'$ its ring of integers, or equivalently, $\mathbf{A}_K':=\mathbf{B}_K'\cap W(\mathbf{C}^\flat)$. Thus, we see that $\mathbf{A}_K'$ is a discrete valuation ring, admitting
$\pi$ as a uniformizer, and that $\mathbf{A}_K'/\pi\mathbf{A}_K'=\mathbf{E}_K'$. In order to emphasize the dependence on $F$, we will sometimes denote $\mathbf{A}_K'$ by $\mathbf{A}_{K|F}'$, and similarly for the related notation. There is a natural lift of the $q$-power Frobenius $\varphi_q$ from $\mathbf{E}_K'$ to $\mathbf{A}_K$. Furthermore the action of $\widetilde{\Gamma}_K$ on $\mathbf{E}_K'=X_K(K_\infty)$ induces an action of $\widetilde{\Gamma}_K$ on $\mathbf{A}_K'$, and this action commutes with that of $\varphi_q$.

As in \cite[\textsection 2.1.9]{EG22}, it will be convenient for us to introduce a variant of the ring $\mathbf{A}_K'$. Let $\Delta_K$ be the torsion subgroup of $\widetilde{\Gamma}_K$. As $\widetilde{\Gamma}_K$ can be identified with an open subgroup of $\O_F^\times$ (via the Lubin--Tate character), we have an isomorphism $\widetilde{\Gamma}_K\cong\Gamma_K\times\Delta_K$, where $\Gamma_K\cong \mathbf{Z}_p^{\oplus [F:\mathbf{Q}_p]}$. We now let $\mathbf{A}_K:=(\mathbf{A}_K')^\Delta_K$.
\begin{lem}
$\mathbf{A}_K$ is again is a complete discrete valuation ring with uniformizer $\pi$ and residue field $\mathbf{E}_K:=(\mathbf{E}_K')^{\Delta_K}$.
\end{lem}
\begin{proof}
Only the statement on residue field needs a proof. We claim that $\widetilde{\Gamma}_K$ (hence $\Delta_K$) acts faithfully on $\mathbf{E}_K'$ (hence on $\mathbf{A}_K'$). Equivalently, we need to show that if $g\in G_K$ acts trivially on $\mathbf{E}_K'$, then necessarily $g\in G_{K_\infty}$. Indeed, in this case we can view $g$ as an element in $G_{\mathbf{E}_K'}$. We can then find some $g'\in G_{K_\infty}$ which maps into $g$ under the Fontaine--Wintenberger isomorphism $G_{K_\infty}\xrightarrow{\sim} G_{\mathbf{E}_K'}$. This means that the actions of $g$ and $g'$ on $(\mathbf{E}_K')^{\mathrm{sep}}$ agree. As $(\mathbf{E}_K')^{\mathrm{sep}}$ is dense in $\mathbf{C}^\flat$ (cf. \cite[Prop. 1.4.27]{Schneider}), they in fact agree on $\mathbf{C}^\flat$, hence $g=g'\in G_{K_\infty}$, as claimed. (We can also prove the claim by first reducing to the case $K=F$, and then use the explicit description of $\mathbf{E}_F'$ as $\mathbf{E}_F'=k((\iota(v))$, where recall that $v$ is a generator of the Tate module $T\F_{\phi}$.)

It now follows from Artin's lemma in Galois theory that $\mathbf{A}_K'$ is finite free of rank $\#\Delta_K$ over $\mathbf{A}_K=(\mathbf{A}_K')^{\Delta_K}$ (e.g. after passing to their fraction fields). In particular, the same is true for the induced extension of residue fields. Similarly, we see that $\mathbf{E}_K'$ has degree $\#\Delta_K$ over $(\mathbf{E}_K')^{\Delta_K}=:\mathbf{E}_K$, and so the latter must coincide with the residue field of $\mathbf{A}_K=(\mathbf{A}_K')^{\Delta_K}$.
\end{proof}
Clearly, $\mathbf{A}_K$ is $\varphi_q$-stable and equipped with an induced action of $\Gamma_K=\widetilde{\Gamma}_K/\Delta_K$. We let $\mathbf{B}_K:=\mathbf{A}_K[1/\pi]$ be its fraction field, so that $\mathbf{B}_K=(\mathbf{B}_K')^{\Delta_K}$.

If $T_K'$ is any lift of a uniformizer of $\mathbf{E}_K'$ to $\mathbf{A}_K'$, then there is an identification $\widehat{W_{\O_F}(k_\infty')((T_K'))}\xrightarrow{\sim}\mathbf{A}_K'$ (again the hat denotes the $\pi$-adic completion), where $k_{K,\infty}'$ denotes the residue field of $K_\infty$. Similarly, for any lift $T_K$ of a uniformizer of $\mathbf{E}_K$, we have an isomorphism $\widehat{W_{\O_F}(k_{K,\infty})((T_K))}\xrightarrow{\sim} \mathbf{A}_K$, where $k_{K,\infty}$ denotes the residue field of the subfield of $K_{\infty}$ corresponding to the subgroup $\Delta_K\subseteq \widetilde{\Gamma}_K$.

In general, it seems difficult to explicitly write down formulas for the actions of $\varphi_q$ and $\widetilde{\Gamma}_K$ on $\mathbf{A}_K'$. This is possible however in the case where $K$ is unramified over $F$ (equivalently, $K=K_0F$). Indeed, we claim that in this case the element $\iota(v)$ defined earlier is also a uniformizer of $\mathbf{E}_K'$. We have seen that it is a uniformizer of $\mathbf{E}_F'$, hence it suffices to show that $\mathbf{E}_K'$ is unramified over $\mathbf{E}_F'$, which is true because $[\mathbf{E}_K':\mathbf{E}_F']=[K_\infty:F_\infty]=[K:K\cap F_\infty]=[K:F]=[k_K:k_F]\leq [k_{K_\infty}:K_F]$. Thus, if we choose $T_K':=\{\iota(v)\}$, then $\mathbf{A}_K'=\widehat{\O_{K_0 F}((T_K'))}$, and the actions of $\varphi_q$ and $\widetilde{\Gamma}_K=\widetilde{\Gamma}_F$ on $T_K'$ are given exactly as before. We furthermore let $T_K:=\prod_{g\in \Delta_F}g\cdot\{\iota(v)\}$ so that $T_K\in \mathbf{A}_K$ (as it is fixed by $\Delta_F=\Delta_K$). By Lemma \eqref{phi-stable when unramified} below, we have $\mathbf{A}_K=\widehat{\O_{K_0 F}((T_K))}$, and furthermore, the ``integral'' subring $\mathbf{A}_K^+:=\O_{K_0 F}[[T_K]]$ is stable under the actions of $\varphi_q$ and $\Gamma_K$.
\begin{lem}\label{phi-stable when unramified}
Assuming $K/F$ is unramified. Then $\mathbf{A}_K=\widehat{\O_{K_0 F}((T_K))}$. Furthermore, the subring $\mathbf{A}_K^+:=\O_{K_0 F}[[T_K]]$ is $(\varphi_q,\Gamma_K)$-stable; in fact, $\varphi_q(T_K)\in T_K\mathbf{A}_K^+$ and $g(T_K)\in T_K\mathbf{A}_K^+$ for all $g\in \Gamma_K$.
\end{lem}
\begin{proof}
For the first statement, it suffices to show that the image $\overline{T}_K$ of $T_K$ in $\mathbf{E}_K'=k_K((\iota(v)))$ is a uniformizer of the residue field $\mathbf{E}_K:=(\mathbf{E}_K')^{\Delta_K}$ of $\mathbf{A}_K$. Since $\mathbf{E}_K'/(\mathbf{E}_K')^{\Delta_K}$ is a totally ramified extension of degree $\#\Delta$,  we only need to show that $|\overline{T}_K|_{\flat}=|\iota(v)|_{\flat}^{\#\Delta}$. To see this, first recall that for each $n\geq 1$, there is a canonical isomorphism $(\O_F/\pi^n)^\times\xrightarrow{\sim} \mathrm{Gal}(F_n/F)$ defined by $a\mapsto \sigma_a:=(x\mapsto [a]_{\phi}(x), x\in F_{\phi}[\pi^n])$. We can now compute 
\begin{align*}
|\overline{T}_K|_{\flat} & = \prod_{a\in\Delta}|\iota(av)|_{\flat}\\
& =\prod_{a\in \Delta}|([a]_{\phi}(v_0)\bmod{\pi}, [a]_{\phi}(v_1)\bmod{\pi},\ldots)|_{\flat}\\
&\stackrel{\text{def}}{=}\prod_{a\in \Delta}\lim_{n} |[a]_{\phi}(v_n)|^{q^n}\\
&=\prod_{a\in \Delta}\lim_{n} |\sigma_a(v_n)|^{q^n}\\
&=\prod_{a\in \Delta}\lim_n |v_n|^{q^n}\\
&\stackrel{\text{def}}{=} |(v_0\bmod{\pi},v_1\bmod{\pi},\ldots)|_{\flat}^{\#\Delta}=|\iota(v)|_{\flat}^{\#\Delta},
\end{align*}
as desired. Set $\mathbf{A}_K^+:=\O_{K_0 F}[[T_K]]$ with $T_K$ defined as above. We check that $\mathbf{A}_K^+$ is stable under the actions of $\varphi_q$ and $\Gamma_K$. As the ring $\O_{K_0 F}[[T_K']]$ is visibly $(\varphi_q,\widetilde{\Gamma}_K)$-stable, it suffices to show that $\mathbf{A}_K\cap \O_{K_0 F}[[T_K']]=\mathbf{A}_K^+$. To see this, let $f\in \mathbf{A}_K\cap \O_{K_0 F}[[T_K']]\subseteq\O_{K_0 F}[[T_K']]$. Then $f\bmod{\pi}\in \mathbf{E}_K\cap\O_{\mathbf{E}_K'}=\O_{\mathbf{E}_K}=k_K[[\overline{T}_K]]$ whence $f=f_0+\pi g_1$ for some $f_0\in \mathbf{A}_K^+$ and $g_1\in \O_{K_0 F}[[T_K']]$. Since $f$ and $f_0$ are both fixed by $\Delta_K$, the same is true for $g_1$. Thus, by induction, we can find a sequence $(f_n)\subseteq \mathbf{A}_K^+$ such that 
\begin{displaymath}
f\equiv f_0+\pi f_1+\ldots+\pi^n f_n\pmod{\pi^{n+1}\O_{K_0 F}[[T_K']]}
\end{displaymath}
for all $n\geq 0$. Since $\O_{K_0 F}[[T_K]]$ and $\O_{K_0 F}[[T_K']]$ are both $\pi$-adically completed, this implies $f\in \mathbf{A}_K^+$, as required. By using the same argument (i.e. consider the reductions modulo $\pi$), we see that $\varphi_q(T_K)\in T_K\mathbf{A}_K^+$ and $g(T_K)\in T_K\mathbf{A}_K^+$ for all $g\in \Gamma_K$.
\end{proof}
For each finite extension $K/F$, we let $K_{\mathrm{cyc}}$ denote the subfield of $K_\infty$ corersponding to the torsion subgroup $\Delta_K$ of $\widetilde{\Gamma}_K$ (it can also be characterized as the unique subextension of $K_\infty/K$ whose Galois groups is isomorphic to $\mathbf{Z}_p^{\oplus [F:\mathbf{Q}_p]}$). The following definition is modeled on \cite[Defn. 2.1.12]{EG22}.
\begin{defn}\label{basic K}
We say that $K$ is $F$-\textit{basic} if it is contained in $(K_0 F)_{\mathrm{cyc}}$, or equivalently, $K_{\mathrm{cyc}}=(K_0 F)_{\mathrm{cyc}}$ (to see this equivalence, it suffices to note that any open subgroup of $\mathbf{Z}_p^{\oplus [F:\mathbf{Q}_p]}$ is itself isomorphic to $\mathbf{Z}_p^{\oplus [F:\mathbf{Q}_p]}$). 
\end{defn}
Thus, if $K$ is $F$-basic, then $\mathbf{A}_K=\mathbf{A}_{K_0 F}$ and the action of ${\Gamma}_K$ is just the restriction of the action of $\Gamma_{K_0 F}$ om $\mathbf{A}_{K_0 F}$. Thus by Lemma \eqref{phi-stable when unramified}, we can choose $T_K:=T_{K_0 F}$ so that the integral subring $\mathbf{A}_K^+:=\O_{K_0 F}[[T_K]]$ is $(\varphi_q,\Gamma_K)$-stable, and moreover $\varphi_q(T_K)\in T_K\mathbf{A}_K^+$ and $g(T_K)\in T_K\mathbf{A}_K^+$ for all $g\in \Gamma_K$ (and similarly for $\mathbf{A}_K'$ and $(\mathbf{A}_K')^+$). In general, as explained above, we can still choose some $T_K$ so that $\mathbf{A}_K=\widehat{W_{\O_F}(k_{K,\infty})((T_K))}$; once this is done, we will set $\mathbf{A}_K^+:=W_{\O_F}(k_{K,\infty})[[T_K]]$. We endow $\mathbf{A}_K^+$ with the $(\pi,T_K)$-adic topology, and $\mathbf{A}_K$ with the unique topology for which a fundamental system of open neighborhoods of $0\in \mathbf{A}_K$ is given by the sets $\pi^n\mathbf{A}_K+T_K^m\mathbf{A}_K^+, n,m\geq 0$\footnote{This is designed so that the quotient topology on the local field $\mathbf{A}_K/\pi\cong k_{K,\infty}((T_K))$ is its valuation topology.} (one can check easily that the subspace topology on $\mathbf{A}_A^+$ is indeed the $(\pi,T_K)$-adic topology).
\begin{remark}
We have seen that if $u\in \mathbf{A}_K$ is any element lifting a uniformizer of the residue field of $\mathbf{A}_K$, then $\mathbf{A}_K$ can be identified with $\widehat{W_{\O_F}(k_{K,\infty})((u))}$. However, for a general choice of $u$, the integral subring $\mathbf{A}_K^+:=W_{\O_F}(k_{K,\infty})[[u]]$ may not be $\varphi_q$-stable (cf. \cite[\textsection 1.1.2.2]{Her98}). As in \cite{EG22}, the advantage of having $\varphi_q$-stability for $\mathbf{A}_K^+$ is that it allows us to invoke the geometric properties proved in \cite{EG19} of (a variant of) the stack $\R_d$ of \' etale $\varphi_q$-modules over $\mathbf{A}_K$ (see Theorem \eqref{R_d} below).
\end{remark}
\subsection{Coefficients}\label{coefficient}
As we will be interested in moduli stacks parametrizing famillies of Lubin--Tate $(\varphi_q,\Gamma_K)$-modules (or related variants), it is necessary to introduce the version of the various rings considered in Subsection \eqref{ring} relative to a varying coefficient ring $A$. As the stacks we consider will liver over the $p$-adic formal scheme $\Spf \mathbf{Z}_p$, the test ring $A$ will be always assumed to be a $p$-adically complete. In fact, as in \cite{EG22}, it will be convenient to introduce an auxiliary base ring over which $A$ lives. More precisely, we will fix a finite extension $E/F$ with ring of integers $\O$, uniformizer $\varpi$ and residue field $\mathbf{F}$; accordingly, $A$ will be taken to be a $\varpi$-adically complete $\O$-algebra. (We will sometimes need to assume $E$ is large enough, e.g. so that it contains all the images of all embeddings $K\hookrightarrow\overline{\mathbf{Q}}_p$, but for now this is irrelevant to us.) 

For $A$ as above, we will set $\mathbf{A}_{K,A}^+:=\mathbf{A}_{K}^+\widehat{\otimes}_{\O_F}A$, where the completed tensor product is taken with respect to the $\varpi$-adic topology on $A$ and the $(\pi,T_K)$-adic topology topology on $\mathbf{A}_{K}^+$ (where as above $T_K$ is a lift of a uniformizer of $\mathbf{E}_K$ to $\mathbf{A}_K^+$ so that we have $\mathbf{A}_K^+=W_{\O_F}(k_{K,\infty})[[T_K]]$), i.e.
\begin{displaymath}
\mathbf{A}_{K,A}^+:=\varprojlim_n\mathbf{A}_K^+/(p,T_K)^n\otimes_{\O_F}A=\varprojlim_m(\varprojlim_n \mathbf{A}_K^+/(p^m,T_K^n)\otimes_{\O_F}A).
\end{displaymath}
Similarly, we define $\mathbf{A}_{K,A}$ to be the completed tensor product $\mathbf{A}_{K}\widehat{\otimes}_{\O_F}A$ taken with respect to the $\pi$-adic topology on $A$ and the natural topology on $\mathbf{A}_K$ (cf. the discussion following Definition \eqref{basic K}). Concretely, we have 
\begin{displaymath}
\mathbf{A}_{K,A}:=\varprojlim_m((\varprojlim_n \mathbf{A}_K^+/(p^m,T_K^n)\otimes_{\O_F}A)[1/T_K]).
\end{displaymath}
Thus, we see that $\mathbf{A}_{K,A}^+\cong (W_{\O_F}(k_{K,\infty})\otimes_{\O_F}A)[[T_K]]$, whereas $\mathbf{A}_{K,A}$ is the $p$-adic completion of $\mathbf{A}_{K,A}^+[1/T_K]$. Similarly, we can define the rings $(\mathbf{A}_{K,A}')^+$ and $\mathbf{A}_{K,A}'$ (use $T_K'$ in place of $T_K$).

Recall that we have the actions of $\varphi_q$ and $\Gamma_K$ on $\mathbf{A}_{K}$. We can extend these actions on each $\mathbf{A}_{K,A}$ by decreeing that they are $A$-linear. In case $\mathbf{A}_K^+$ is $(\varphi_q,\Gamma)$-stable (e.g. when $K$ is $F$-basic), the same is true of the ring $\mathbf{A}_{K,A}^+$ for each $A$. For a more detailed discussion on these actions, see \cite[Prop. 2.2.17]{EG22} (which in turn rests on \cite[Lem. B.31]{EG22} and \cite[Lem. B.34]{EG22}).    
\subsection{The relationship with Galois representations}
Denote by $\mathbf{A}_K^{\mathrm{ur}}\subseteq W_{\O_F}({\mathbf{C}}^\flat)$ the $\pi$-adic completion of the maximal unramified extension of $\mathbf{A}_K'$ (or equivalently of $\mathbf{A}_F'$) in $W_{\O_F}({\mathbf{C}}^\flat)$. As $\mathbf{A}_K'$ is preserved by the natural actions of $\varphi_q$ and $G_K$ on $W_{\O_K}(\O_{\mathbf{C}}^\flat)$, the same is true for the ring $\mathbf{A}_K^{\mathrm{ur}}$ by functoriality.

Let $(\mathbf{E}_K')^{\mathrm{sep}}$ be the separable closure of $\mathbf{E}_K'$ inside $\mathbf{C}^\flat$. It is clear that the natural action of $G_K$ on $\mathbf{C}^\flat$ preserves $(\mathbf{E}_K')^{\mathrm{sep}}$. If, in addition, $g\in \mathrm{Gal}(\overline{K}/K_{\infty})$, then $g$ fixes $\widehat{K_\infty}^\flat$, and in particular, $\mathbf{E}_K'$. In this way, we obtain a natural homomorphism
\begin{displaymath}
G_{K_\infty}\to G_{\mathbf{E}_K'}
\end{displaymath}
which turns out to be an isomorphism by the theory of norm fields of Fontaine--Wintenberger. From this, we can deduce the following result.
\begin{thm}\label{equivalence with Galois rep no coefficient}
The functor 
\begin{displaymath}
V\mapsto D(V):=(V\otimes_{\O_F}\mathbf{A}_K^{\mathrm{ur}})^{G_{K_\infty}}
\end{displaymath}
gives an equivalence of categories between the category $\mathrm{Rep}_{\O_F}(G_K)$ of continuous $G_K$-representations on finite $\O_F$-modules and the category of finite \' etale $(\varphi_q,\widetilde{\Gamma}_K)$-modules over $\mathbf{A}_K'$. A quasi-inverse functor is given by $M\mapsto T(M):=(\mathbf{A}_K^{\mathrm{ur}}\otimes_{\mathbf{A}_K'}M)^{\varphi_q=1}$. 

Furthermore, $V$ is free of rank $d$ over $\O_F$ if and only if $D(V)$ is free of rank $d$ over $\mathbf{A}_K'$.
\end{thm}
\begin{proof}
This is proved in \cite[Thm. 1.6]{KisinRen}, see also \cite[Thm. 3.3.10]{Schneider} for a more detailed exposition. (Strictly speaking, the latter reference only considers the case $K=F$, but the proof for general $K$ follows the same strategy.)
\end{proof}
In fact, there is also an analogous equivalence of categories using $(\varphi_q,\Gamma_K)$-modules, and taking $\Delta_K$-invariants gives an equivalence of categories between $(\varphi_q,\widetilde{\Gamma}_K)$-modules
and $(\varphi_q,\Gamma_K)$-modules (see Lemma \eqref{Delta invariant} below). Since it is the version with $\mathbf{A}_K$-coefficients that we will work mostly with, let us record the equivalence for $(\varphi_q,\Gamma_K)$-modules over $\mathbf{A}_K$ separately below. 
\begin{thm}
(Recall that $K_{\mathrm{cyc}}$ denotes the subextension of $K_\infty$ corresponding to the subgroup $\Delta_K$ of $\widetilde{\Gamma}_K$.) The functor
\begin{displaymath}
V\mapsto D(V):=(V\otimes_{\O_F}\mathbf{A}_K^{\mathrm{ur}})^{G_{K_{\mathrm{cyc}}}}
\end{displaymath}
gives an equivalence of categories between the category $\mathrm{Rep}_{\O_F}(G_K)$ of continuous $G_K$-representations on finite $\O_F$-modules and the category of finite \' etale $(\varphi_q,{\Gamma}_K)$-modules over $\mathbf{A}_K$. A quasi-inverse functor is given by $M\mapsto T(M):=(\mathbf{A}_K^{\mathrm{ur}}\otimes_{\mathbf{A}_K}M)^{\varphi_q=1}$.

Furthermore, $V$ is free of rank $d$ over $\O_F$ if and only if $D(V)$ is free of rank $d$ over $\mathbf{A}_{K}$.
\end{thm}
\begin{lem} \label{Delta invariant}
The functor 
\begin{displaymath}
M\mapsto M':=\mathbf{A}_K'\otimes_{\mathbf{A}_K}M
\end{displaymath}
gives an equivalence of categories between finite \' etale $(\varphi_q,\Gamma_K)$-module over $\mathbf{A}_K$ and finite \' etale $(\varphi_q,\widetilde{\Gamma}_K)$-module over $\mathbf{A}_K'$. A quasi-inverse functor is given by $M'\mapsto (M')^{\Delta}$.

Again, $M$ is free of rank $d$ over $\mathbf{A}_K$ if and only if $M'$ is free of rank $d$ over $\mathbf{A}_K'$.
\end{lem}
\begin{proof}
We first show that the natural map $(M')^\Delta\otimes_{\mathbf{A}_K}\mathbf{A}_K'\to M'$ is an isomorphism for any finite \' etale $(\varphi_q,\widetilde{\Gamma}_K)$-module $M'$ over $\mathbf{A}_K'$ (the induced $\varphi_q$-action on $M^\Delta$ is then automatically \' etale as its linearization becomes an isomorphism after base changing along the faithfully flat map $\mathbf{A}_K\hookrightarrow \mathbf{A}_K'$.)

As $M'\xrightarrow{\sim}\varprojlim_nM'/\pi^n$, we may reduce to the case where $M'$ is killed by $\pi^n$ for some $n\geq 1$ (as $\mathbf{A}_K'$ is finite free over $\mathbf{A}_K$, tensoring with $\mathbf{A}_K'$ commutes with inverse limits). The case $n=1$ follows directly from Hilbert's 90 theorem. Note also that $H^1(\Delta,M'')=0$ in this case. Now assume $n\geq 1$ and $M'$ is killed by $\pi^{n+1}$. Let $M'':=\pi^nM'$ and $M''':=M'/M''$ so that we have an exact sequence $0\to M''\to M'\to M'''\to 0$. As $M''$ is killed by $\pi$, $H^1(\Delta,M'')=0$, and the sequence $0\to (M'')^\Delta\to (M')^\Delta\to (M''')^\Delta_K\to 0$ remains exact. Thus, we obtain a commutative diagram with exact rows:
\begin{displaymath}
\begin{tikzcd}
0 \ar[r] & \mathbf{A}_K'\otimes_{\mathbf{A}_K}(M'')^\Delta\ar[d] \ar[r] & \mathbf{A}_K'\otimes_{\mathbf{A}_K}(M')^\Delta\ar[d]\ar[r] & \mathbf{A}_K'\otimes_{\mathbf{A}_K} (M''')^\Delta\ar[d]\ar[r] & 0\\
0\ar[r] & M''\ar[r] & M'\ar[r] & M'''\ar[r] & 0.
\end{tikzcd}
\end{displaymath}
By induction, the outer vertical maps are isomoprhisms, and hence the same is true of the middle map.

Using a similar inductive argument, we can show that for any finite \' etale $(\varphi_q,\Gamma_K)$-module $M$ over $\mathbf{A}_K$, the natural injection $M\hookrightarrow M':=\mathbf{A}_K'\otimes_{\mathbf{A}_K}M$ identifies $M$ with $(M')^{\Delta}$ (injectivity follows from the fact that $\mathbf{A}_K'$ is faithfully flat over $\mathbf{A}_K$). The last statement is clear as any finite $\pi$-torsion free module over the discrete valuation ring $\mathbf{A}_K$ is necessarily free.
\end{proof}
\section{Moduli stacks of \texorpdfstring{$\varphi_q$}{phi}-modules and Lubin-Tate \texorpdfstring{$(\varphi_q,{\Gamma}_K)$}{(phi,Gamma)}-modules}
\subsection{Moduli stacks of \texorpdfstring{$\varphi_q$}{phi}-modules}\label{stacks of phi-modules}
In this subsection, we briefly define the moduli stacks of \' etale $\varphi_q$-modules, and show in particular that they are Ind-algebraic stacks (Theorem \eqref{R_d}). 
\begin{setup}
Fix a finite extension $l\supseteq k\supseteq \mathbf{F}_p$. let $\mathbf{A}^+:=W(l)[[T]]$, and let $\mathbf{A}$ be the $p$-adic completion of $\mathbf{A}^+[1/T]$. Let $\varphi$ be a ring endomorphism of $\mathbf{A}$ which is trivial on the subring $W(k)\subseteq W(l)$, and moreover congruent to the $q$-power Frobenius modulo $p$ for some fixed power $q$ of $p$.  
\end{setup}
As before, we will fix a finite extension $E/W(k)[1/p]$ with uniformizer $\varpi$ and ring of integers $\O$, which will serve as the base for our test rings. If $A$ is a $\varpi$-adically complete $\O$-algebra, we set $\mathbf{A}_{A}^+:=(W(l)\otimes_{W(k)}A)[[T]]$, equipped with the $(p,T)$-adic topology. We also let $\mathbf{A}_A$ be the $p$-adic completion of $\mathbf{A}_A^+[1/T]$, and endow it with the unique topology for which a fundamental system of neighborhoods of $0$ is given by the sets $p^n\mathbf{A}_A+T^m\mathbf{A}_A^+, n,m\geq 0$ (again one checks easily that the subspace topology on $\mathbf{A}_A^+$ is indeed the $(p,T)$-adic topology). In particular, we see that the rings $\mathbf{A}_{K,A}^+, \mathbf{A}_{K,A}$ introduced earlier are special case of the this construction (with $l=k_{K,\infty}$ being the residue field of $K_{\mathrm{cyc}}$ and $k$ being the residue field of $F$). 

Again, as in \cite[Prop. 2.2.17]{EG22}, we can extend $\varphi$ to a continuous $A$-linear action on $\mathbf{A}_{A}$. In case $\mathbf{A}^+$ is moreover $\varphi$-stable (i.e. $\varphi(\mathbf{A}^+)\subseteq \mathbf{A}^+$), the same is true of each $\mathbf{A}_A^+$.
\begin{remark}
Our setting here is slightly more general than that considered in \cite[Chap. 5]{EG19} as we are considering power series (or Laurent series) over rings of the form $W(l)\otimes_{W(k)}A$ (as opposed to the ring $W(l)\otimes_{\mathbf{Z}_p}A$ in that reference). Nevertheless, as we will see shortly, we can deduce various results about our stacks from the corresponding results in that paper.
\end{remark}
\begin{defn}
Let $A$ be a $\varpi$-adically complete $\O$-algebra. An \' etale $\varphi$-module with $A$-coefficients is a finitely generated $\mathbf{A}_A$-module, equipped with a $\varphi$-semilinear morphism $\varphi_M: M\to M$ for which the linearized map $\Phi_M: \varphi^*M\to M$ is an isomorphism.
\end{defn}
\begin{defn}
Fix an integer $d\geq 1$. For each $a\geq 1$, let $\R_{d}^{a}$ be the fibered category over $\Sp(\O/\varpi^a)$ taking an $\O/\varpi^a$-algebra $A$ to the groupoid of rank $d$ projective \' etale $\varphi$-modules with $A$-coefficients. It follows from Drinfeld's descent results (cf. \cite[Thm. 5.1.18]{EG19}) that $\R_{d}^{a}$ is in fact an \textit{fpqc} (hence \textit{fppf}) stack over $\O/\varpi^a$. As usual, we may regard $\R_{d}^{a}$ as an \textit{fppf} stack over $\Sp \O$ (equipped with a map to $\Sp (\O/\varpi^a)$). We then define $\R_{d}$ as the colimit of the stacks $\R_{d}^{a}$ over all $a\geq 1$. Thus $\R_{d}$ is an \textit{fppf} stack over $\Sp\O$ which furthermore admits a (necessarily unique) morphism to $\Spf\O$. Moreover, using \cite[Prop. 0.7.2.10(ii)]{EGAInew} we see that for any $\varpi$-adically complete $\O$-algebra $A$ (not just those living over some $\O/\varpi^a$), $\R_{d}(A)$ (which, by definition, is the groupoid of morphisms $\Spf A\to \R_{d}$) is equivalent to the groupoid of rank $d$ projective \' etale $\varphi$-modules with $A$-coefficients.
\end{defn}
We now define the moduli stacks of so-called finite height $\varphi$-modules. For this, we need to assume that $\mathbf{A}^+$ is $\varphi$-stable. We also need to fix a polynomial $F$ in $(W(l)\otimes_{W(k)}\O)[[T]]$ which is congruent to a positive power of $T$ modulo $\varphi$ (e.g. $F$ can be $T$ itself, or an Eisenstein polynomial). 
\begin{defn}
Assume $\mathbf{A}^+$ is $\varphi$-stable. Let $h$ be a non-negative integer, and let $A$ be a $\varpi$-adically complete $\O$-algebra. A $\varphi$-module of $F$-height $\leq h$ with $A$-coefficients is a finitely generated $T$-torsion free $\mathbf{A}_A^+$-module $\fM$, equipped with a $\varphi$-semilinear morphism $\varphi_{\fM}: \fM\to \fM$ for which the linearized map $\Phi_{\fM}: \varphi^*\fM\to \fM$ is injective, and has cokernel killed by $F^h$.
\end{defn}
\begin{defn}
Fix $d\geq 1$ and $h\geq 0$. For each $a\geq 1$, let $\C_{d,h}^{a}$ be the \textit{fppf} stack over $\Sp(\O/\varpi^a)$ taking an $\O/\varpi^a$-algebra $A$ to the groupoid of rank $d$ projective $\varphi$-modules of $F$-height $\leq h$ with $A$-coefficients. Again, the colimit $\C_{d,h}:=\varinjlim_{a}\C_{d,h}^{a}$ is an \textit{fppf} stack over $\Spf\O$, whose groupoid of $A$-valued points, for any $\varpi$-adically complete $\O$-algebra $A$, is equivalent to the groupoid of rank $d$ projective $\varphi$-modules of $F$-height $\leq h$ with $A$-coefficients.
\end{defn} 
For each $a\geq 1$, as $T$ is invertible in $\mathbf{A}_A$, the same is true of $F$; thus, we have a natural morphism $\C_{d,h}^a\to \R_d^a$ taking $\fM\to \fM[1/T]$. Taking the colimit over $a\geq 1$, we obtain a map $\C_{d,h}\to \R_d$.
\begin{thm}\label{R_d}
Assume that $\mathbf{A}^+$ is $\varphi$-stable. 
\begin{itemize}
    \item[\emph{(1)}] $\C_{d,h}$ is a $p$-adic formal algebraic stack of finite presentation over $\Spf \O$, with affine diagonal.
    \item[\emph{(2)}] The morphism $\C_{d,h}\to \R_d$ is representable by algebraic spaces, proper, and of finite presentation.
    \item[\emph{(3)}] $\R_{d}$ is a limit preserving Ind-algebraic stack, whose diagonal is representable by algebraic spaces, affine, and of finite presentation.  
\end{itemize}
\end{thm}
\begin{proof}
By \cite[Prop. A.13]{EG22}, it suffices to show that $\C_{d,h}^a$ is an algebraic stack of finite presentation over $\Sp \O/\varpi^a$, with affine diagonal. That $\C_{d,h}^a$ is an algebraic stack of finite type over $\Sp \O/\varpi^a$ is \cite[Thm. 2.1 (a)]{PR09}. More precisely, the arguments in that reference remain valid in our setting provided that we replace the group $G$ there by $\mathrm{Res}_{W(l)/W(k)}\mathrm{GL}_d$, that we replace \cite[Prop. 2.2]{PR09} by \cite[Lem. 5.2.9]{EG19}, and accordingly, that in \cite[\textsection 3]{PR09} we replace $eah$ by the quantity $n(a,h)$ in \cite[Lem. 5.2.9]{EG19}. The point of \cite[Lem. 5.2.9]{EG19} is that when $n>n(a,h)$, one can replace the $\varphi$-conjugation action of $U_n:=1+T^n M_d(\mathbf{A}_A^+)$ on $LG^{\leq h}:=\{A\in \mathrm{GL}_d(\mathbf{A}_A)\;|\;A^{\pm 1}\in T^{-h}M_d(\mathbf{A}_A^+)\}$ by the \textit{free} action of $U_n$ given by left translations. Together with the fact that any object in $\C_{d,h}^a(M)$ admits a basis locally on $\Sp A$ (cf. \cite[Lem. 5.1.9 (1)]{EG19}), one can then show that $\C_{d,h}^a$ can be written as the quotient stack of a finite type $\O/\varpi^a$-scheme by the action of a smooth finite type group scheme over $\O/\varpi^a$, and so is an algebraic stack of finite type over $\O/\varpi^a$ (e.g. by \cite[\href{https://stacks.math.columbia.edu/tag/06FI}{Tag 06FI}]{Sta21}).

With the same modifications, the arguments in \cite[Thm. 2.5 (b)]{PR09} show that the natural morphism $\C_{d,h}^a\to \R_{d,\mathrm{\textit{fpqc}-free}}^a, \fM\mapsto \fM[1/T]$ is representable by algebraic spaces, proper, and of finite presentation, where $\R_{d,\mathrm{\textit{fpqc}-free}}^a$ is the substack of $\R_d^a$ classifying those objects which are furthermore free \textit{fpqc} locally on $\Sp A$. Combining this with the arguments in \cite[Lem. 5.4.10]{EG19}, \cite[Thm. 5.4.11]{EG19}, we deduce that the map $\C_{d,h}^a\to \R_d^a$ is representable by algebraic spaces, proper, and of finite presentation. At this point, we have shown (1) and (2), except the claim that $\C_{d,h}$ has affine diagonal (indeed, once this is done, it will follow that $\C_{d,h}$ has quasi-compact and quasi-separated diagonal, hence is quasi-separated and then of finite presentation over $\O/\varpi^a$). As the map $\C^a_{d,h}\to \R_d^a$ is representable by algebraic spaces, and proper (hence separated), its diagonal is a proper monomorphism, hence a closed immersion. Using this, we are reduced to show that the diagonal of $\R_d^a$ is affine. This will be done in part (3) below.

(3) When $W(k)=\mathbf{Z}_p$, this is one of the main results of \cite{EG19}. As alluded earlier, our strategy for the general case is to relate $\R_d^a$ with the ``corresponding'' stack introduced in \cite{EG19}. More precisely, let $\R_d^{\mathrm{EG},a}$ be the \textit{fppf} stack over $\Sp \O/\varpi^a$, taking an $\O/\varpi^a$-algebra $A$ to the groupoid of rank $d$ projective \' etale $\varphi$-modules over $(W(l)\otimes_{\mathbf{Z}_p}A)((T))$. Clearly, we have $\R_d^{\mathrm{EG},a}(A)\cong \R_d^a(W(k)\otimes_{\mathbf{Z}_p}A)$. Combining with the $A$-algebra isomorphism $W(k)\otimes_{\mathbf{Z}_p} A\cong \prod_{\overline{\sigma}: k\hookrightarrow \overline{\mathbf{F}}_p}A$, it follows that there is an isomorphism $\R_d^{\mathrm{EG},a}\cong \prod_{\overline{\sigma}}\R_d^a$ of stacks over $\O/\varpi^a$.

Assume for the moment that the diagonal of $\R_d^a$ is affine, and of finite presentation. By induction and the standard graph argument, we deduce that the diagonal $\Delta_n: \R_d^{a}\to \prod_{1\leq i\leq n}\R_d^a$ has the same properties for all $n\geq 1$ (factor $\Delta_n$ as $\Delta_{n+1}$ followed by the projection $\prod_{n+1}\R_d^a\twoheadrightarrow \prod_n \R_d^a$ onto the first $n$ components, and note that the latter map, being a base change of the structure map $\R_d^a\to \Sp(\O/\varpi^a)$, also has finitely presented affine diagonal). In particular, we obtain a morphism $\R_d^a\to \R_d^{\mathrm{EG},a}$ which is affine, and of finite presentation. Now it follows from \cite[Cor. 3.2.9]{EG22}, \cite[Thm. 5.4.11]{EG19} and \cite[Thm. 5.4.20]{EG19} that $\R_d^a$ has the claimed properties.  

Thus, it remains to show that the diagonal of $\R_d^a$ is affine, and of finite presentation. \textit{Mutatis mutandis}, this is proved in \cite[Prop. 5.4.8]{EG19}.
\end{proof}
\subsection{Moduli stacks of Lubin-Tate \texorpdfstring{$(\varphi_q,{\Gamma}_K)$}{(phi,Gamma)}-modules}
We can now define our main objects of interest. We will keep the notation as in Setup \eqref{situation}; in particular we will fix a finite extension $F/\mathbf{Q}_p$, a Lubin--Tate formal group law associated to a uniformizer $\pi$ of $F$, and a finite extension $K$ of $F$. As in Subsection \eqref{coefficient}, we will also fix throughout a finite extension $E/F$ with uniformizer $\varpi$ and ring of integers $\O$, which will serve as the base of our coefficients $A$. Recall also that by an \' etale $(\varphi_q,{\Gamma}_K)$-module with $A$-coefficients, we mean an \' etale $(\varphi_q,{\Gamma}_K)$-module over $\mathbf{A}_{K,A}$ in the usual sense.
\begin{defn}
Fix an integer $d\geq 1$. We let $\X^{\lt}_{K,d}$ denote the \textit{fppf} stack over $\Spf \O$, whose groupoid of $A$-valued points, for any $\varpi$-adically complete $\O$-algebra $A$, is equivalent to the groupoid of rank $d$ \' etale $(\varphi_q,{\Gamma}_K)$-modules with $A$-coefficients. (That this is well-defined follows exactly as in the definitions of the stacks $\R_d$ and $\C_{d,h}$; in particular we have implicitly used Drinfeld's descent results for verifying the stack property.)
\end{defn}
\subsubsection{Basic geometric properties of \texorpdfstring{$\X_{K,d}^{\lt}$}{Xd}}
We want to show the following preliminary result regarding the geometry of the stack $\X_{K,d}^{\lt}$. A more detailed study of its geometric properties will be given in Section \eqref{finer geometric}; in particular, we will show that it is in fact a Noetherian formal algebraic stack (Corollary \eqref{Xd noetherian formal}). 
\begin{thm}\label{Xd ind algebraic}
$\X_{K,d}^{\lt}$ is a limit preserving Ind-algebraic stack whose diagonal is affine (in particular, representable by schemes), and of finite presentation.
\end{thm}
For this we will follow closely the argument used in \cite{EG22} for the stack $\X_{K,d}$ of usual (cyclotomic) \' etale $(\varphi,\Gamma)$-modules; in particular, the strategy is to deduce the claimed properties for $\X_{K,d}^{\lt}$ from the corresponding properties of the stack of \' etale $\varphi_q$-modules. Thus, let us first define the latter. 
\begin{defn}
We let $\R_{K,d}$ be the moduli stack of \' etale $\varphi_q$-modules over $\mathbf{A}_K$; in other words, it is the stack $\R_d$ defined in Subsection \eqref{stacks of phi-modules}, with $\mathbf{A}$ there taken to be $\mathbf{A}_{K}$. 
\end{defn}
In case the ring $\aka^+$ is $\varphi_q$-stable (e.g. if $K$ is $F$-basic in the sense of Definition \eqref{basic K}), we can apply Theorem \eqref{R_d} to use various properties of the stack $\R_{K,d}$. Our first goal is to show that these properties still hold for general $K$ (i.e. without assuming that $\mathbf{A}_K^+$ is $\varphi_q$-stable). 
\begin{defn}[{{\cite[Defn. 3.2.3]{EG22}}}]
We set $K^{\mathrm{basic}}:=K\cap (K_0 F)_{\mathrm{cyc}}$.
\end{defn}
It is easy to see that $K^{\mathrm{basic}}$ is indeed $F$-basic and $(K^{\mathrm{cyc}})_{\mathrm{cyc}}=(K_0 F)_{\mathrm{cyc}}$, and that the natural map $\Gamma_{K^{\mathrm{basic}}}\hookrightarrow \Gamma_{K}$ is an isomorphism. Furthermore, as $\mathbf{A}_{K}$ is finite free of rank $[K_{\mathrm{cyc}}:K^{\mathrm{basic}}_{\mathrm{cyc}}]=[K:K^{\mathrm{basic}}]$ over $\mathbf{A}_{K^{\mathrm{basic}}}$, we have a natural forgetful map $\R_{K,d}\to \R_{K^{\mathrm{basic},d[K:K^{\mathrm{basic}}]}}$.
\begin{lem}\label{K to K basic Rk}
The natural map $\R_{K,d}\to \R_{K^{\mathrm{basic}},d[K:K^{\mathrm{basic}}]}$ is affine, and of finite presentation.
\end{lem}
\begin{proof}
Again the proof of \cite[Lem. 3.2.5]{EG22} works \textit{mutatis mutandis} in our setting.
\end{proof}
\begin{cor}
For any $K/F$, $\R_{K,d}$ is a limit preserving Ind-algebraic stack,
whose diagonal is affine, and of finite presentation.
\end{cor}
\begin{proof}
Combine Theorem \eqref{R_d}, Lemma \eqref{K to K basic Rk}, and \cite[Cor. 3.2.9]{EG22}.
\end{proof}
Having shown that $\R_{K,d}$ has the claimed properties, the next step is to define and study an ``intermediate" stack sitting between $\R_{K,d}$ and $\X_{K,d}$, namely, it will be a stack of \' etale $\varphi_q$-modules which are furthermore equipped with an action of certain ``discretization" of $\Gamma_K$. More precisely, let $\gamma_1,\ldots,\gamma_{[F:\mathbf{Q}_p]}$ be a $\mathbf{Z}_p$-free basis for $\Gamma_K\cong\mathbf{Z}_p^{\oplus [F:\mathbf{Q}_p]}$, and accordingly, $\Gamma_\disc\subseteq \Gamma_K$ be the sub-$\mathbf{Z}$-module generated by $\gamma_1,\ldots,\gamma_{[F:\mathbf{Q}_p]}$. Let $\R_d^{\Gamma_{\disc}}$ be the moduli stack of \' etale $\varphi_q$-modules over $\mathbf{A}_{K,A}$ equipped with a (not necessarily continuous!) semilinear action of $\Gamma_\disc$ that commutes with $\varphi$. 
\begin{lem}\label{R_d Gamma_disc as a fixed point stack}
The natural morphism $\R_d^{\Gamma_\disc}\to \R_d$ given by forgetting the $\Gamma_\disc$-action, is affine, and of finite presentation.
\end{lem}  
\begin{proof}
We can prove the statement after pulling back along a morphism $\Sp A\to \R_d$ where $A$ is a $\mathbf{Z}/p^a$-algebra for some $a\geq 1$. Denote by $M$ the \' etale $\varphi_q$-module over $\mathbf{A}_{K,A}$ corresponding to this morphism. Note firstly that the data of a commuting semi-linear action of $\Gamma_\disc$ on $M$ is precisely the data of isomorphisms $\alpha_i: \gamma_i^*M\xrightarrow{\sim} M$ of $\varphi_q$-modules for $1\leq i\leq [F:\mathbf{Q}_p]$ such that $\alpha_i\circ \gamma_i^*\alpha_j=\alpha_j\circ\gamma_j^*\alpha_i$ for all $i,j$. Thus we need to show that the functor on $A$-algebras sending $B$ to the set 
\begin{displaymath}
\left\{(\alpha_i)_i\in \prod_{i}\mathrm{Isom}_{\mathbf{A}_{K,B},\varphi_q}(\gamma_i^*M_B,M_B)\;|\;\alpha_i\circ \gamma_i^*\alpha_j=\alpha_j\circ\gamma_j^*\alpha_i\;\text{for all $i,j$}\right\}
\end{displaymath}
is represented by an affine $A$-scheme of finite presentation. By \cite[Prop. 5.4.8]{EG19}, the functor on $A$-algebras sending $B$ to $\prod_{i}\mathrm{Isom}_{\mathbf{A}_{K,B},\varphi_q}(\gamma_i^*M_B,M_B)$ is represented by an affine scheme of finite presentation over $A$. More precisely, as in the proof of that proposition, we may reduce to the case where $M_B$ is finite free over $\mathbf{A}_{K,B}$; then after choosing bases, any morphism $\alpha_i: \gamma_i^*M\xrightarrow{\sim} M$ of $\varphi_q$-modules is determined by the coefficients of finitely many powers of $T$ in the Laurent series expansions of the entries of the matrix representing $\alpha_i$. To conclude, it suffices to note that, the conditions $\alpha_i\circ \gamma_i^*\alpha_j=\alpha_j\circ\gamma_j^*\alpha_i$ are evidently given by finitely many equations in these coefficients.
\end{proof}
\begin{cor}\label{diagonal of R_d Gamma_disc}
$\R_d^{\Gamma_\disc}$ is a limit preserving Ind-algebraic stack whose diagonal is affine, and of finite presentation.
\end{cor}
\begin{proof}
This follows from Lemma \eqref{R_d Gamma_disc as a fixed point stack} and \cite[Lem. 3.2.9]{EG22}.
\end{proof}
Restricting the $\Gamma_K$-action from an \' etale $(\varphi_q,\Gamma_K)$-module to $\Gamma_\disc$ defines a fully faithful morphism $\X_{K,d}^{\lt}\to \R_d^{\Gamma_\disc}$. Indeed, by construction, any linear map between two finite projective $\mathbf{A}_{K,A}$-modules is continuous with respect to the canonical topology (cf. \cite[Exam. 3.2.2]{Dri06}, \cite[Rem. D.2]{EG22}). As $\Gamma_{\disc}$ is dense in $\Gamma_K$, it follows that any $\Gamma_{\disc}$-equivariant morphism between two objects of $\X_{K,d}^{\lt}$ is automatically $\Gamma_K$-equivariant. 
\begin{cor}\label{diagonal of Xk}
The diagonal of $\X_{K,d}^{\lt}$ is representable by algebraic spaces, affine, and of finite presentation.
\end{cor}
\begin{proof}
As the map $\X_{K,d}^{\lt}\to \R_d^{\Gamma_\disc}$ is a monomorphism, its diagonal is an isomorphism. Using the cartesian diagram
\begin{displaymath}
    \begin{tikzcd}
\X_{K,d}^{\lt}\times_{\R_d^{\Gamma_\disc}}  \X_{K,d}^{\lt}\ar[d]\ar[r] & \X_{K,d}^{\lt}\times_{\Spf \O}  \X_{K,d}^{\lt}\ar[d]\\
\R_d^{\Gamma_\disc} \ar[r] & \R_d^{\Gamma_\disc}\times_{\Spf \O}\R_d^{\Gamma_\disc}
    \end{tikzcd}
\end{displaymath}
we see that the diagonal of $\X_d^{\lt}$ is a base change of that of $\R_d^{\Gamma_\disc}$, which has the claimed properties by Corollary \eqref{diagonal of R_d Gamma_disc}. 
\end{proof}
The next lemma allows us to reduce to the case where $K$ is $F$-basic (in proving Theorem \eqref{Xd ind algebraic}).
\begin{lem}\label{cartesian diagram xk rk xk basic}
We have a Cartesian diagram
\begin{displaymath}
\begin{tikzcd}
\X_{K,d}^{\lt}\ar[d,hook] \ar[r] & \X_{K^{\mathrm{basic}},d[K:K^{\mathrm{basic}}]}^{\lt}\ar[d,hook]\\
\R_{K,d}^{\Gamma_\disc}\ar[r] & \R_{K^{\mathrm{basic}},d[K:K^{\mathrm{basic}}]}^{\Gamma_{\disc}},
\end{tikzcd}
\end{displaymath}
where the horizontal arrows are the natural forgetful maps, and the vertical arrows are the monomorphisms given by restricting
the action of $\Gamma_K$ to $\Gamma_{\mathrm{disc}}$.
\end{lem}
\begin{proof}
This is clear because $\Gamma_{\disc}$ acts continuously on an object $M$ of $\R_{K,d}^{\Gamma_{\disc}}$ only if it does so when $M$ is regarded as an object of $\R_{K^{\mathrm{basic}},d[K:K^{\mathrm{basic}}]}^{\Gamma_{\disc}}$.
\end{proof}
\begin{cor}\label{Xk to Xkbasic}
The natural map $\X_{K,d}^{\lt}\to \X_{K^{\mathrm{basic}},d[K:K^{\mathrm{basic}}]}^{\lt}$ is affine, and of finite presentation.
\end{cor}
\begin{proof}
This follows easily from Lemmas \eqref{R_d Gamma_disc as a fixed point stack}, \eqref{K to K basic Rk}, and \eqref{cartesian diagram xk rk xk basic}.
\end{proof}
In view of Corollaries \eqref{diagonal of Xk}, \eqref{Xk to Xkbasic} and \cite[Cor. 3.2.9]{EG22}, it remains to show that if $K=K^{\mathrm{basic}}$ is $F$-basic, then $\X_{K,d}^{\lt}$ is a limit preserving Ind-algebraic stack. For this we will need to understand the continuity condition in the definition of an object of $\X_{K,d}^{\lt}$ more carefully.

The following is an analogue of \cite[Lem. 3.2.18]{EG22} in our setting.
\begin{lem}\label{3.2.18}
For any $\gamma\in \Gamma_K$, we have $\gamma(T)-T\in \pi\mathbf{A}_{K,A}+T^2\mathbf{A}_{K,A}^+$. If $K$ is $F$-basic, then $\gamma(T)-T\in (\pi,T)T\mathbf{A}_{K,A}^+$.
\end{lem}
\begin{proof}
The proof is identical to that of \cite[Lem. 3.2.18]{EG22}, except we consider the reduction modulo $\pi$ instead of modulo $p$ (recall also that in case $K$ is $F$-basic, we have chosen $T$ so that $\gamma(T)\in T\mathbf{A}_{K,A}^+$ for all $\gamma\in \Gamma_K$).
\end{proof}
Using Lemma \eqref{3.2.18}, we can obtain (among other things) an analogue of \cite[Lem. D.28]{EG22} in our setting. This will be worked out in Appendix \ref{T quasi linear} below. In particular, we can now apply these results to study the action of $\Gamma_K$ on the objects of $\X_{K,d}^{\lt}$.
\begin{lem}\label{X limit preserving}
The stack $\X_{K,d}^{\lt}$ is limit preserving.
\end{lem}
\begin{proof}
We may assume that $K$ is $F$-basic. As the natural map $\X_{K,d}^{\lt}\hookrightarrow \R_{K,d}^{\Gamma_\disc}$ is fully faithful with limit preserving target, it suffices to show that the map is limit preserving on objects. In other words, given any direct limit $A=\varinjlim_{i\geq i_0}A_i$ of $\O/\varpi^a$-algebra for some $a\geq 1$ and any object $M$ of $\R_{K,d}^{\Gamma_{\disc}}(A_{i_0})$ such that $M_A:=M\otimes_{\ak{A_{i_0}}}\ak{A}$ lies in the subgroupoid $\X_{K,d}^{\lt}(A)$, we need to find some $i\geq i_0$ such that $M_{A_i}\in \X_{K,d}^{\lt}(A_i)$, i.e. that the action of $\Gamma_{\disc}$ on $M_{A_i}$ extends (uniquely) to a continuous action of $\Gamma_K$. This follows from the Lemma \eqref{continuous action Gamma disc} below and (the proof of) \cite[Lem. D.31]{EG22} (note also the obvious observation that the action of $\Gamma_\disc$ is continuous if and if its restriction to each $\langle \gamma_i\rangle$ is continuous).
\end{proof}
\subsubsubsection{\texorpdfstring{$\X_{K,d}^{\lt}$}{Xd} is an Ind-algebraic stack}
We now turn to showing that $\X_{K,d}^{\lt}$ is an Ind-algebraic stack (thereby completing our proof of Theorem \eqref{Xd ind algebraic}). As in \cite{EG22}, we will need to relate $\X_{K,d}^{\lt}$ with the moduli stacks of weak Wach modules. We will assume throughout that $K$ is $F$-basic.  
\begin{defn}[{{\cite[Def. 3.3.2]{EG22}}}]
(Assume $K$ is $F$-basic.) A rank $d$ projective weak Wach module of $T$-height $\leq h$ with $A$-coefficients is a rank $d$ projective $\varphi_q$-module $\fM$ over $\mathbf{A}_{K,A}^+$ of $T$-height $\leq h$, such that $\fM[1/T]$ is equipped with a continuous semi-linear action of $\Gamma_K$.

If $s\geq 0$, then we say that $\fM$ has level $\leq s$ if $(\gamma_i^{p^s}-1)(\fM)\subseteq T\fM$ for all $1\leq i\leq n$.
\end{defn}
Denote by $\W_{d,h}$ the moduli stack of weak Wach modules which are of $T_K$-height $\leq h$. For each $s\geq 0$, denote by $\W_{d,h,s}$ the substack of those weak Wach modules of level $\leq s$. (Again, by using Drinfeld's descent results, we see that these are \textit{fppf} stacks over $\Spf\O$.) By Lemma \eqref{continuous action Gamma disc} below, any projective weak Wach module is of level $\leq s$ for some $s\gg 0$.

Let $\C_{d,h}$ denote the moduli stack of rank $d$ projective $\varphi_q$-modules of $T$-height $\leq h$ over $\mathbf{A}_{K}^+$ (that is, it is the stack $\C_{d,h}$ defined in Subsection \eqref{stacks of phi-modules} with $\mathbf{A}$ taken to be $\mathbf{A}_{K,A}$). By Theorem \eqref{R_d}, this is a $p$-adic formal algebraic stack of finite presentation over $\Spf\O$. Consider the fiber product $\R_d^{\Gamma_{\disc}}\times_{\R_d}\C_{d,h}$ where the map $\R_d^{\Gamma_\disc}\to \R_d$ is given by forgetting the $\Gamma_\disc$-action. This is the moduli stack of rank $d$ projective $\varphi_q$-modules $\fM$ of $T$-height $\leq h$, equipped with a semilinear action of $\Gamma_\disc$ on $\fM[1/T]$. As the map $\R_d^{\Gamma_\disc}\to \R_d$ is representable by algebraic spaces, and of finite presentation (by Lemma \eqref{R_d Gamma_disc as a fixed point stack}), the same is true of the map $\R_d^{\Gamma_\disc}\times_{\R_d}\C_{d,h}\to \C_{d,h}$. It follows that $\R_d^{\Gamma_\disc}\times_{\R_d}\C_{d,h}$ is also a $p$-adic formal algebraic stack of finite presentation over $\Spf \O$.

By restricting the $\Gamma_K$-action to $\Gamma_\disc$, we obtain a natural morphism $\W_{d,h}\to \R_d^{\Gamma_\disc}\times_{\R_d}\C_{d,h}$ which is again fully faithful since $\Gamma_\disc$ is dense $\Gamma_K$.
\begin{prop}\label{w d h s}
For each $s\geq 0$, the natural morphism
\begin{displaymath}
\W_{d,h,s}\to \R_d^{\Gamma_\disc}\times_{\R_d}\C_{d,h}
\end{displaymath} 
is a closed immersion of finite presentation. In particular, each $\W_{d,h,s}$ is a $p$-adic formal algebraic stack of finite presentation over $\Spf \O$, and $\W_{d,h}=\varinjlim_{s}\W_{d,h,s}$ is an Ind-algebraic stack. 
\end{prop}
\begin{proof}
See the proof of \cite[Lem. 3.3.5]{EG22}.
\end{proof}
For each $a\geq 1$, let $\W_{d,h,s}^a:=\W_{d,h,s}\times_{\Spf \O}\Sp \O/\varpi^a$, an algebraic stack of finite presentation over $\O/\varpi^a$. As $\R_d^{\Gamma_\disc}$ is an Ind-algebraic stack, it makes sense to consider the scheme-theoretic image $\X_{d,h,s}^a$ of the composite
\begin{displaymath}
\W_{d,h,s}^a\hookrightarrow \W_{d,h,s}\to \R_d^{\Gamma_\disc},
\end{displaymath}
As the map $\W_{d,h,s}^a\to \R_d^{\Gamma_\disc}$ is representable by algebraic spaces, proper and of finite presentation by Proposition \eqref{w d h s} (recall that the natural map $\C_{d,h}\to \R_d$ has the same properties by Theorem \eqref{R_d}), we find that $\X_{d,h,s}^a$ is a closed algebraic substack of $\R_d^{\Gamma_\disc}$ which is of finite presentation over $\Sp \O/\varpi^a$, and moreover, the induced map $\W_{d,h,s}^a\to \X_{d,h,s}^a$ is proper, scheme-theoretically dominant and surjective. 
\begin{prop}
Each $\X_{d,h,s}^a$ is a closed (algebraic) substack of $\X_{K,d}^{\lt}$. Moreover, the induced map $\varinjlim_{a,h,s}\X_{d,h,s}^a\to \X_{K,d}^{\lt}$ is an isomorphism. Thus $\X_{K,d}^{\lt}$ is an Ind-algebraic stack, and may in fact be written as the inductive limit of algebraic stacks of finite presentation over $\Sp \O$, with the transition maps being closed immersions.
\end{prop}
\begin{proof}
The arguments given in \cite[Lem. 3.4.8, 3.4.9 and 3.4.10]{EG22} go through provided one replaces the condition ``$(\gamma^{p^s}-1)(\fM)\subseteq T\fM$'' with ``$(\gamma_i^{p^s}-1)(\fM)\subseteq T\fM$ for all $i$" everywhere. For convenience of the reader, we sketch the proof here. We begin with the first assertion. It suffices to show that $\X_{d,h,s}^a$ is a substack of $\X_{K,d}^{\lt}$. By definition, this amounts to showing that if $M$ is an \' etale $\varphi$-module with coefficients in a finite type $\O/\varpi^a$-algebra $A$, equipped with a $\Gamma_\disc$-action for which the map $\Sp A\to \R_d^{\Gamma_\disc}$ classifying $M$ factors through $\X_{d,h,s}^a$, then the $\Gamma_\disc$-action on $M$ is continuous. In view of Lemma \eqref{continuous action Gamma disc}, we need to produce a lattice $\fM\subseteq M$ such that $(\gamma_i^{p^s}-1)\fM\subseteq T\fM$ for all $i$. By embedding $A$ into the product of its Artinian quotients, we may reduce to the case where $A$ is a finite Artinian $\O/\varpi^a$-algebra. In this case, the existence of the desired lattice is established in \cite[Lem. 3.4.8]{EG22} provided one modifies the definition of the subfunctor $D'$ there by imposing the condition that $(\gamma_i^{p^s}-1)\fM_{\Lambda}\subseteq T\fM_{\Lambda}$ for all $i$.

For the second assertion, we need to show that any morphism $\Sp A\to \X_{K,d}^{\lt}$ whose source is a Noetherian $\O/\varpi^a$-algebra necessarily factors through some $\X_{d,h,s}^a$. It suffices to do this after replacing $A$ by some algebra $B$ for which the map $\Sp B\to \Sp A$ is scheme-theoretically dominant. Thus, by \cite[Lem. 5.4.7]{EG19}, we may assume that $M$ is free, where $M$ is the \' etale $\varphi$-module over $\aka$ classified by the map $\Sp A\to \R_d$. In particular, we can choose a free $\varphi$-stable lattice $\fM$ inside $M$, say of height $\leq h$ for some $h$ large enough. As the $\Gamma$-action on $M$ is continuous by definition, Lemma \eqref{continuous action Gamma disc} implies that for all $s$ sufficiently large, we have $(\gamma_i^{p^s}-1)\fM\subseteq T\fM$ for all $i$. It follows that the map $\Sp A\to \X_{K,d}^{\lt}$ classifying $M$ factors through $\W_{d,h,s}^a$, and hence through $\X_{d,h,s}^a$, as required. 
\end{proof}
As mentioned before, using Corollary \eqref{Xk to Xkbasic}, we can now drop the assumption that $K$ is $F$-basic. 
\begin{prop}\label{Xd ind algebraic K general}
Let $K$ be an arbitrary finite extension of $F$. Then $\X_{K,d}^{\lt}$
is a limit preserving Ind-algebraic stack, which can in fact be written as the inductive limit of a sequence of algebraic stacks of finite presentations over $\Sp \O$ with transition maps being closed immersions. Furthermore the diagonal of $\X_{K,d}^{\lt}$ is affine (in particular, representable by schemes), and of finite presentation.
\end{prop}
\subsection{Galois representations with coefficients}
As one might expect from the case of cyclotomic $(\varphi,\Gamma)$-modules, there is also an analogue of the equivalence in Theorem \eqref{equivalence with Galois rep no coefficient} in the presence of coefficients. 

To this end, let $A$ be a complete local Noetherian $\O$-algebra with residue field. As in \cite[\textsection 3.6.1]{EG22}, we denote by $\widehat{\mathbf{A}}_{K,A}$ the $\m_A$-adic completion of $\aka$, and we define
a formal \' etale $(\varphi_q,\Gamma_K)$-module with $A$-coefficients to be an \' etale $(\varphi_q,\Gamma_K)$-module over $\widehat{\mathbf{A}}_{K,A}$ in the obvious sense. We also let $\widehat{\mathbf{A}}^{\mathrm{ur}}_{K,A}$ be the $\m_A$-adic completion of the tensor product $\mathbf{A}^{\mathrm{ur}}\otimes_{\O_F}A$.
\begin{thm}\label{Dee equivalence}
The functor 
\begin{displaymath}
V\mapsto D_A(V):=(V\otimes_A\widehat{\mathbf{A}}_{K,A}^{\mathrm{ur}})^{G_{K_\infty}}
\end{displaymath}
gives an equivalence of categories between the category of continuous $G_K$-representations on finite free $A$-modules and the category of finite projective formal \' etale $(\varphi_q,\widetilde{\Gamma}_K)$-modules with $A$-coefficients. A quasi-inverse functor is given by $M\mapsto T_A(M):=(\widehat{\mathbf{A}}_{K,A}^{\mathrm{ur}}\otimes_{\widehat{\mathbf{A}}_{K,A}} M)^{\varphi_q=1}$.
\end{thm}
\begin{proof}
This is a generalization of of \cite[Thm. 2.2.1]{Dee01} to the setting of Lubin--Tate $(\varphi_q,\Gamma_K)$-modules. The idea is to first consider the case where $V$ has finite length. In this case $D_A(V)$ simplifies to $D_{\O_F}(V)$ where in the latter we view $V$ as a module over $\O_F$; in particular, we can make use of the already-established case of $\O_F$-linear representations (Theorem \eqref{equivalence with Galois rep no coefficient}). Once this is done, we pass to the limit to deduce the general case (after checking that everything behave well under taking limits). A detailed proof has been recently worked out in \cite[Thm. 7.18]{HerrBerger}.
\end{proof}
The above equivalence also holds in the case where $A=\overline{\mathbf{F}}_p$ (this follows easily from the fact that $\X_{K,d}^{\lt}$ is limit preserving). In other words, we see that the groupoid of $\overline{\mathbf{F}}_p$-points of $\X_{K,d}^{\lt}$ is equivalent to the groupoid of representations $G_K\to\mathrm{GL}_d(\overline{\mathbf{F}}_p)$. Using this equivalence, we can deduce easily that taking versal rings of the stack $\X_{K,d}^{\lt}$ at finite type points recovers the usual framed Galois deformation rings.

As in the cyclotomic setting, we can also define the universal unramified rank one \' etale $(\varphi_q,\Gamma_K)$-module.

Recall that $k_F\cong \mathbf{F}_q$ (resp. $k_K\cong \mathbf{F}_{q^r}$) denotes the residue field of $F$ (resp. $K$). Unless otherwise stated, we assume from now on that the base ring $\O$ is large enough so that it contains $W_{\O_F}(k_K)$ (or equivalently, that $\mathbf{F}$ contains $k_K$).

\begin{lem}
Let $A$ be an $\O$-algebra, and let $a\in A^\times$. Then, up to isomorphism, there is a unique free \' etale $\varphi_q$-module  $D_{K,a}$ of rank one over $W_{\O_F}(k_K)\otimes_{\O_F}A$ with the property that $\varphi_q^r=1\otimes a$ on $D_{K,a}$.
\end{lem}
\begin{proof}
We need to show that the norm map $(W_{\O_F}(k_K)\otimes_{\O_F}A)^\times\to A^\times, x\mapsto N(x):=x\varphi_q(x)\ldots \varphi_q^{r-1}(x)$ is surjective with kernel given by the set $\{\varphi_q(y)/y\;|\;y\in (W_{\O_F}(k)\otimes_{\O_F}A)^\times\}$. As $\O$ (and hence $A$) is naturally a $W_{\O_F}(k_K)$-algebra, we have an $A$-algebra isomorphism $W_{\O_F}(k_K)\otimes_{\O_F}A\to \prod A, x\otimes 1\mapsto (x,\varphi_q(x),\ldots,\varphi_q^{r-1}(x))$. The lemma then follows easily from this ismorphism. 
\end{proof}
\begin{defn}
Let $A$ be a $\varpi$-adically complete $\O$-algebra, and let $a\in A^\times$. Define $\aka(\ur{a}):=D_{K,a}\otimes_{W_{\O_F}(k_K)\otimes_{\O_F}A}\ak{A}$. This is a rank one \' etale $(\varphi_q,\Gamma_K)$-module with $A$-coefficients, where we let $\varphi_q$ act diagonally, and $\Gamma_K$ act on the second factor. 
\end{defn}
\begin{lem}\label{universal unramified}
Let $A$ be a finite Artinian local $\O$-algebra, and let $a\in A^\times$. Then, under the equivalence \eqref{Dee equivalence}, $\aka(\ur{a})$ corresponds to the unramified character $\ur{a}$ of $G_K$ sending geometric Frobenii to $a$.
\end{lem}
\begin{proof}
The case $F=\mathbf{Q}_p$ is proved e.g. in \cite[Lem. 2.4]{Pham}; since the proof is short, we will recall the argument here (which works also for general $\mathbf{F}$). To this end, note that the rank one $A$-representation of $G_{K}$ corresponding to $\aka(\ur{a})$ can be computed as
\begin{align*}
    V &:= ((W_{\O_F}(\mathbf{C}^\flat) \widehat{\otimes}_{\O_F}A)\otimes_{\aka}\aka(\ur{a}))^{\varphi_q=1}\\
    &\cong \{hv\;|\; h\in W_{\O_F}(\mathbf{C}^\flat)\otimes_{\O_F}A\;\text{such that $\varphi_q(hv)=hv$}\},
\end{align*}
where $v$ is a basis of $D_{K,a}$. Assume $V$ has a basis $hv$ with $h\in W_{\O_F}(\overline{\mathbf{F}}_p)\otimes_{\O_F}A$. We claim that $G_{K}$ acts on this basis via the unramified character taking geometric Frobenii to $a$. First, for $\sigma\in I_K$, we have $\sigma(hv)=\sigma(h)v=hv$ (note that $\sigma(h)=h$ by our assumption on $h$). Now, let $\sigma$ be an arithmetic Frobenius in $G_K$. By design, $\varphi_q(hv)=hv$ and $\varphi_q^r(v)=av$, so $hv=\varphi_q^r(hv)=\varphi_q^r(h)av=\sigma(h)av$, and thus $\sigma(hv)=\sigma(h)v=a^{-1}(hv)$, as claimed. 

It remains to find the desired basis. If $A$ is a field, say $A=\mathbf{F}'$ for some finite extension $\mathbf{F}'/\mathbf{F}$, this follows easily from the ring isomorphism $\mathbf{C}^\flat\otimes_{\mathbf{F}_q}\mathbf{F}'\xrightarrow{\sim} \prod \mathbf{C}^\flat$. In general, by induction, we may assume that, for some ideal $I$ with $I^2=0$, there is a basis $\overline{h}v$ of $V/I V$ with $\overline{h}\in W_{\O_F}(\overline{\mathbf{F}}_p)\otimes_{\O_F}(A/I)$. Let $h\in W_{\O_F}(\overline{\mathbf{F}}_p)\otimes_{\O_F}A$ be any lift of $\overline{h}$. Then $\varphi_q(hv)=g(hv)$ for some $g\in 1+W_{\O_F}(\overline{\mathbf{F}}_p)\otimes_{\O_F} I$, say $g=1+g_1\otimes m_1+\ldots+g_n\otimes m_n$ with $g_i\in W_{\O_F}(\overline{\mathbf{F}}_p)$ and $m_i\in I$. By a standard sucessive approximation argument, we can find, for each $i$, an element $h_i\in W_{\O_F}(\overline{\mathbf{F}}_p)$ such that $\varphi_q(h_i)-h_i=-g_i$. If $f:=1+h_1\otimes m_1+\ldots h_n\otimes m_n\in 1+W_{\O_F}(\overline{\mathbf{F}}_p)\otimes_{\O_F} I$, then $\varphi_q(f)=f/g$, and hence $\varphi_q(hfv)=hfv$. As $(hf)v$ lifts a basis of $V/I V$, it is necessarily a basis of $V$ by Nakayama's lemma.
\end{proof}
\section{Finer geometric properties of \texorpdfstring{$\X_{K,d}^{\lt}$}{Xd}}\label{finer geometric}
\subsection{Relation with the Emerton--Gee stack}
Having established that $\X_{K,d}^{\lt}$ is an Ind-algebraic stack, our goal in this subsection is to make a more detailed study of its geometry. More precisely, we will show that our stack of (rank $d$) Lubin--Tate \' etale $(\varphi_q,\Gamma_K)$-modules is in fact isomorphic to the Emerton-Gee stack of (rank $d$) cyclotomic $(\varphi,\Gamma)$-modules. As an immediate consequence, we deduce that $\X_{K,d}^{\lt}$ is in fact a Noetherian formal algebraic stack. 

We maintain our notation from the previous section. Our main result is the following.
\begin{prop}\label{equivalence LT EG}
Assume $A$ is a finite type $\mathbf{\O}/\varpi^a$-algebra for some $a\geq 1$. Then there is an exact tensor-compatible rank-preserving equivalence between the category of finite projective \' etale $(\varphi_q,\Gamma_K)$-modules over $\aka$, and the category of finite projective \' etale $(\varphi,\Gamma_K^{\eg})$-modules over $\aka^{\eg}$.
\end{prop}
Here, as the notation suggests, $\aka^{\eg}$ is the coefficient ring for the cyclotomic $(\varphi,\Gamma)$-modules considered in \cite{EG22} (it is denoted by $\aka$ in \textit{loc. cit.}). Also, exactness here means that both the equivalence and its inverses are exact.
\begin{proof}
By \cite[Prop. 2.7.8]{EG22}, if $A$ is as in the statement, then the functor $M\mapsto M\otimes_{\aka^{\eg}}(A\widehat{\otimes}_{\mathbf{Z}_p}W(\mathbf{C}^\flat))$ defines an exact equivalence between finite projective \' etale $(\varphi,\Gamma^{\eg}_K)$-modules over $\aka^{\eg}$, and finite projective \' etale $(\varphi,G_K)$-modules over $A\widehat{\otimes}_{\mathbf{Z}_p}W(\mathbf{C}^\flat)$ (exactness follows from fully faithfulness of the map $\aka^{\eg}\to A\widehat{\otimes}_{\mathbf{Z}_p}W(\mathbf{C}^\flat)$, cf. \cite[Prop. 2.2.12]{EG22}). In the same way, extending scalars along $\aka=A\widehat{\otimes}_{\O_F}\mathbf{A}_K\to A\widehat{\otimes}_{\O_F}W_{\O_F}(\mathbf{C}^\flat)$ defines an exact equivalence between the category of finite projective \' etale $(\varphi_q,\Gamma_K)$-modules over $\aka$, and finite projective \' etale $(\varphi_q,G_K)$-modules over $A\widehat{\otimes}_{\O_F}W_{\O_F}(\mathbf{C}^\flat)$ (\textit{loc. cit.} is written only for the ring of unramified Witt vectors, but the proof works in general).

It thus suffices to relate \' etale $\varphi$-modules over $A\widehat{\otimes}_{\mathbf{Z}_p}W(\mathbf{C}^\flat)$, and \' etale $\varphi_q$-modules over $A\widehat{\otimes}_{\O_F}W_{\O_F}(\mathbf{C}^\flat)$. To this end, let $\sigma_0: k_F\hookrightarrow \mathbf{C}^\flat$ be the canonical embedding. For each $j\in \mathbf{Z}/f\mathbf{Z}$, let $\sigma_j:=\sigma_0\circ \varphi^j$, where $\varphi$ is the $p$-power Frobenius on $k_F$. As $W(k_F)\otimes_{\mathbf{Z}_p}W(\mathbf{C}^\flat)\xrightarrow{\sim}\prod_j W(\mathbf{C}^\flat), a\otimes x\mapsto (\sigma_j(a) x)_j$, we obtain an isomorphism
\begin{align*}
    A\widehat{\otimes}_{\mathbf{Z}_p}W(\mathbf{C}^\flat) & = A\widehat{\otimes}_{W(k_F)}(W(k_F)\otimes_{\mathbf{Z}_p} W(\mathbf{C}^\flat))\\
    &\xrightarrow{\sim} \prod_j A\widehat{\otimes}_{W(k_F),\sigma_j}W(\mathbf{C}^\flat)\\
    &\xrightarrow{\sim} \prod_j A \widehat{\otimes}_{\O_F,\sigma_j}W_{\O_F}(\mathbf{C}^\flat).
\end{align*}
Thus, any module $M$ over $A\widehat{\otimes}_{\mathbf{Z}_p}W(\mathbf{C}^\flat)$ decomposes as $M=\prod_j M_j$, where $M_j$ is the base change of $M$ along the map $p_j: A\widehat{\otimes}_{\mathbf{Z}_p}W(\mathbf{C}^\flat) \twoheadrightarrow A \widehat{\otimes}_{\O_F,\sigma_j}W_{\O_F}\mathbf{C}^\flat)$. We claim that the functor $M\mapsto M_0$ defines an equivalence between the category of \' etale $\varphi$-modules over $A\widehat{\otimes}_{\mathbf{Z}_p}W(\mathbf{C}^\flat)$, and the category of \' etale $\varphi_q$-module over $A \widehat{\otimes}_{\O_F,\sigma_0}W_{\O_F}(\mathbf{C}^\flat)$ (as the map $p_0$ is clearly $\Gamma_K$-equivariant, we then obtain an equivalence between the corresponding categories of \' etale $(\varphi,G_K)$-modules). 

Let us abusively denote also by $\varphi$ the ring map $A \widehat{\otimes}_{\O_F,\sigma_j}W_{\O_F}(\mathbf{C}^\flat)\to A \widehat{\otimes}_{\O_F,\sigma_{j+1}}W_{\O_F}(\mathbf{C}^\flat)$ induced by the $p$-power Frobenius on $\mathbf{C}^\flat$. As $\varphi\circ p_{j}=p_{j+1}\circ \varphi$, the isomorphism $\Phi_{M}: \varphi^*M\xrightarrow{\sim} M$ induces an isomorphism $\varphi^*M_{j}\xrightarrow{\sim}M_{j+1}$ for each $j$. In particular, we have an isomorphism $\varphi_q^*M_0=(\varphi^f)^*M_0\xrightarrow{\sim} M_0$, i.e. $M_0$ is an \' etale $\varphi_q$-module. Conversely, given such $M_0$, we can define $M$ to be the module over $A\widehat{\otimes}_{\mathbf{Z}_p}W(\mathbf{C}^\flat)$ corresponding to the tuple $((\varphi^{j})^*M_0)_j$. By design, there is a linear isomorphism $\varphi^*M\xrightarrow{\sim} M$ as $(\varphi^*M)_{j+1}=\varphi^*(M_{j})=\varphi^*((\varphi^{j})^*M_0)=(\varphi^{j+1})^*M_0=M_{j+1}$ for each $j$. In other words, $M$ is an \' etale $\varphi$-module, as desired.
\end{proof}
\begin{cor}\label{equivalence LT EG stacks}
There is an isomorphism
\begin{align*}
    \X^{\lt}_{K,d}\xrightarrow{\sim} \X^{\eg}_{K,d}.
\end{align*}
\end{cor}
\begin{proof}
This follows immediately from Proposition \eqref{equivalence LT EG} as $\X^{\lt}_{K,d}$ and $\X^{\eg}_{\lt}$ are both limit preserving by Lemma \eqref{X limit preserving} and \cite[Lem. 3.2.19]{EG22}.
\end{proof}
As an immediate consequence of the above comparison result, we obtain the following refinement of Theorem \eqref{Xd ind algebraic} on the geometry of the stack $\X_{K,d}^{\lt}$. See also Subsection \eqref{families of extensions} below for a related discussion.
\begin{cor}\label{Xd noetherian formal}
$\X_{K,d}^{\lt}$ is a Noetherian formal algebraic stack over $\Spf \O$. The underlying reduced substack $\X_{d,\mathrm{red}}^{\lt}$ is an algebraic stack of finite presentation over $\mathbf{F}$. Moreover, the irreducible components of $\X_{d,\mathrm{red}}^{\lt}$ admits a natural labeling by Serre weights. 
\end{cor}
\begin{proof}
This follows from Corollary \eqref{equivalence LT EG stacks}, and the corresponding properties of the stack $\X_{K,d}^{\eg}$, see \cite[Cor. 5.5.18]{EG22} and \cite[Thm. 6.5.1]{EG22}.
\end{proof}
\subsection{The Lubin--Tate Herr complex}\label{herr complex section}
In this subsection we introduce a version of the Herr complex for Lubin--Tate $(\varphi_q,\Gamma_K)$-modules with coefficients; the goal is to show that it is a perfect complex. Although this is not strictly needed in our proof that $\X_{K,d}^\lt$ is a Noetherian formal algebraic stack, the result itself may be of independent interest.

Again, we will keep the notation in Section \eqref{definition related}. In particular, we fix a set $\{\gamma_1,\ldots,\gamma_n\}$ of topological generators of the group $\Gamma_K\cong \mathbf{Z}_p^{\oplus [F:\mathbf{Q}_p]}$.

Let $A$ be a $\varpi$-adically complete $\O$-algebra, and let $M$ be a finite projective \' etale $(\varphi_q,\Gamma_K)$-module with $A$-coefficients. The Herr complex $\C^\bullet(M)$ of $M$ is by definition the cohomological Koszul complex of $M$ with respect to the commuting operators $\varphi_q-1, \gamma_1-1,\ldots,\gamma_n-1$. Concretely, if we let $\gamma_0:=\varphi_q$, then $\C^\bullet(M)$ is by the complex
\begin{displaymath}
M \xrightarrow{d^0} \bigoplus_{0\leq i_1\leq n} M \xrightarrow{d^1} \bigoplus_{0\leq i_1<i_2\leq n} M\to \ldots \bigoplus_{0\leq i_1<\ldots<i_{n}\leq n}\xrightarrow{d^n} M
\end{displaymath}
sitting in (cohomological) degrees $0,1,\ldots,n+1$, where for each $0\leq r\leq n$, the component $d^r{|_{i_1<\ldots<i_r}^{j_1<\ldots<j_{r+1}}}: M\to M$ of the $r$th differential $d^r$ is the multiplication by
\begin{displaymath}
\begin{cases} 0\hspace{1.9cm}\quad\text{if $\{i_1,\dots,i_r\}\nsubseteq \{j_1,\ldots,j_{r+1}\}$}\\
(-1)^s(\gamma_j-1)\quad\text{if $\{j_1,\ldots,j_{r+1}\}=\{i_1,\dots,i_r\}\coprod\{j\}$},\end{cases}
\end{displaymath}
where $s$ is the number of elements in the set $\{i_1<\ldots<i_r\}$, which are smaller than $j$.
\begin{example}
If $n=2$, then $\C^\bullet(M)$ can be identified with the complex
\begin{displaymath}
[M\xrightarrow{\begin{psmallmatrix} \varphi_q-1 \\ \gamma_1-1 \\ \gamma_2-1\end{psmallmatrix}} M^{\oplus 3}\xrightarrow{\begin{psmallmatrix} -(\gamma_1-1) & \varphi_q-1 & 0\\ -(\gamma_2-1) & 0 & \varphi_q-1\\0 & -(\gamma_2-1) & (\gamma_1-1) \end{psmallmatrix}} M^{\oplus 3} \xrightarrow{\begin{psmallmatrix} \gamma_2-1 & -(\gamma_1-1) & \varphi_q-1\end{psmallmatrix}} M].
\end{displaymath}
\end{example}
In what follows, unless otherwise stated, we assume that $A$ is an $\O/\varpi^a$-algebra for some $a\geq 1$. We have the following generalization of \cite[Lem. 5.1.2]{EG22}, which allows us to interpret the cohomology groups of the Herr complex entirely in terms of the ambient category of \' etale $(\varphi_q,\Gamma_K)$-modules.
\begin{prop}\label{Herr cohomology and ext}
There is a natural isomorphism of cohomological $\delta$-functors $H^i(\C^\bullet(\cdot)))\xrightarrow{\sim} \ext_{\lt}^i(\aka,\cdot)$.
\end{prop}
Let us first recall the definition of the extension groups $\ext_{\lt}^i(\aka,M)$. Let $\ltet$ be the category of finite projective \' etale $(\varphi_q,\Gamma_K)$-modules over $\aka$. This is an exact full additive subcategory of the abelian category of all \textit{finite} \' etale $(\varphi_q,\Gamma_K)$-modules over $\aka$ (in the sense that if $0\to M\to N\to P\to 0$ is a short exact sequence of finite \' etale $(\varphi_q,\Gamma_K)$-module with $M, P$ finite projective, then the same is true for $N$). 

As is the case for any exact category, given any object $M$ in $\ltet$ and any $i\geq 1$, we have the abelian group $\mathrm{Ext}_{\lt}^i(\aka,M)$ of equivalence classes of so-called degree $i$ Yoneda extensions of $\aka$ by $M$ in $\ltet$. By definition, such an extension is an exact sequence 
\begin{displaymath}
E: 0\to M\to M_{i-1} \to M_{i-2}\to \ldots \to M_0\to \aka\to 0
\end{displaymath}
in $\ltet$\footnote{This is to say that the sequence is exact when viewed as a sequence in the \textit{abelian} category of all finite \' etale $(\varphi_q,\Gamma_K)$-modules over $\aka$.}, and the relevant equivalence relation is generated by the relation that identifies two extensions $E$ and $E'$ whenever there is a map of extensions $E\to E'$, i.e. a commutative diagram
\begin{displaymath}
\begin{tikzcd}
0 \ar[r] & M\ar[d,"\mathrm{id}",swap] \ar[r] & M_{i-1}\ar[d]\ar[r]\ar[r] & \ldots \ar[r] & M_0\ar[d]\ar[r] & \aka\ar[d,"\mathrm{id}"] \ar[r] & 0\\
0 \ar[r] & M \ar[r] & M'_{i-1}\ar[r] & \ldots \ar[r] & M'_0\ar[r] & \aka \ar[r] & 0.
\end{tikzcd}
\end{displaymath}
Concretely,  $E$ and $E'$ are equivalent if and only if there is some extension $E''$ together with maps of extensions $E\xleftarrow{} E''\xrightarrow{} E'$ as above. 

It is easy to see that the abelian group $\mathrm{Ext}^i_{\lt}(\aka,M)$ is naturally an $A$-module with $a\in A$ acting via the map $a: \mathrm{Ext}^i_{\lt}(\aka,M)\to \mathrm{Ext}^i_{\lt}(\aka,M)$ induced by the multiplication by $a$ on $M$ (or on $\aka$). As usual, for $i=0$, we define $\ext^0_{\lt}(\aka,M)$ to be the $A$-module of morphisms $\aka\to M$ in $\ltet$. 

The collection of functors $(\mathrm{Ext}^i_{\lt}(\aka,\cdot))_{i\geq 0}$ forms a (cohomological) $\delta$-functor from the exact category $\M_{\varphi,\Gamma}^{\mathrm{\acute{e}t},\lt}$ to the abelian category of $A$-modules. More precisely, given any short exact sequence $0\to M\to N\to P\to 0$, there is an associated long exact sequence for Yoneda extensions
\begin{align}\label{long exact ext}
    \ldots \to \ext^{i}_{\lt}(\aka,M)\to \ext^{i}_{\lt}(\aka,N)\to \ext^i_{\lt}(\aka,P)\xrightarrow{\delta} \ext^{i+1}_{\lt}(\aka,M) \to \ldots,
\end{align}
where the connecting map $\delta$ is defined by ``splicing'' a degree $i$ extension $[0\to P\to Z_{i-1}\to \ldots \to Z_0\to \aka\to 0]$ in $\ext_{\lt}^i(\aka,P)$ with the given short exact sequence $0\to M\to N\to P\to 0$. A proof of this can be found in \cite[Chap. VII, Thm. 5.1]{mitchell} (strictly speaking, the reference works throughout with an abelian category, but the proof goes over unchanged to any exact category).  

Similarly, as the association $M\mapsto \C^\bullet(M)$ is clearly functorial and exact in $M$, we see easily that the collection $(H^i(\C^\bullet(\cdot)))_{i\geq 0}$ also forms a $\delta$-functor between the same categories. 

In what follows by an embedding $M\hookrightarrow N$ between objects in $\ltet$, we will mean an equivariant injective map $M\to N$ which splits on the level of underlying $\aka$-modules (so that the quotient $P:=N/M$ will also be an object in $\ltet$\footnote{In a general exact category, such embedding is often called a strict (or admissible) monomorphism.}).  
\begin{lem}\label{effaceable ext}
For each $i\geq 1$, the functor $M\mapsto \ext_{\lt}^i(\aka, M)$ is effaceable. In particular, the collection $(\ext_{\lt}^i(\aka,\cdot))_{i\geq 0}$ forms a universal $\delta$-functor.
\end{lem}
\begin{proof}
Given any $M\in \M_{\varphi,\Gamma}^{\mathrm{\acute{e}t},\lt}$, and any class $E\in \mathrm{Ext}^i_{\lt}(\aka,M)$, we need to find an embedding $M\hookrightarrow N$ in $\M_{\varphi,\Gamma}^{\mathrm{\acute{e}t},\lt}$ so that $E=0$ in $\ext_{\lt}^i(\aka,N)$. If $[0\to M\to N\to Z_{i-2}\to \ldots \to Z_0\to \aka\to 0]$ is an extension representing $E$, then $E=\delta(E')$ by construction of the connecting map $\delta$, where $E'$ is the class in $\ext_{\lt}^{i-1}(\aka,P)$ represented by the extension $[0\to P\to Z_{i-2}\to \ldots \to Z_0\to \aka\to 0]$ (as above, we set $P:=\mathrm{Im}(N\to Z_{i-2})$, an object in $\ltet$). Using the long exact sequence \eqref{long exact ext}, we find that $E=0$ in $\ext_{\lt}^i(\aka,N)$, as desired.
\end{proof}
\begin{lem}\label{effaceable lubin tate}
For each $i\geq 1$, the functor $M\mapsto H^i(\C^\bullet(M))$ is effaceable. In particular, the collection $(H^i(\C^\bullet(\cdot)))_{i\geq 0}$ forms a universal $\delta$-functor.
\end{lem}
\begin{proof}
Given any class $c$ in $H^i(\C^\bullet(M))$, we need to find an embedding $M\hookrightarrow N$ in $\M_{\varphi,\Gamma}^{\mathrm{\acute{e}t},\lt}$ so that $c$ becomes zero in $H^i(\C^\bullet(N))$. For brevity, in this proof and the next two lemmas, by ``enlarging'' $M$, we mean replacing $M$ by an appropriate object $N$ in $\M_{\varphi,\Gamma}^{\mathrm{\acute{e}t},\lt}$ for which there is an embedding $M\hookrightarrow N$ as in the statement.

There is nothing to prove if $i>n+1$. Now consider the case $i=1$. We need to show that given any tuple $(x_0,x_1,\ldots,x_n)\in M^{\oplus \binom{n+1}{1}}$ such that $(\gamma_j-1)(x_k)=(\gamma_k-1)(x_j)$ for all $0\leq j<k\leq n$, where $\gamma_0:=\varphi_q$, we can find $N$ and some $x\in N$ so that $(\gamma_j-1)(x)=x_j$ for all $0\leq j\leq n$. This is clear: we can just set $N:=M\oplus \aka x$ with $\gamma_j x:=x+x_j$ for each $0\leq j\leq n$. We can check easily that our assumption on the tuple $(x_j)_j$ guarantees precisely that these actions pairwise commute, and so define an \' etale $(\varphi_q,\Gamma_K)$-module structure on $N$, as wanted.

The case $i=n+1$ follows directly from Lemma \eqref{(phi - 1) M} below. It remains to treat the case where $2\leq i\leq n$. In this case, given any tuple $(y_{j_0 j_1\ldots j_{i-1}})_{0\leq j_0<\ldots<j_{i-1}\leq n}\in M^{\oplus \binom{n+1}{i}}$ such that
\begin{align}\label{dy=0}
    \sum_{k=0}^i (-1)^k(\gamma_k-1) y_{j_0\ldots \hat{j}_k\ldots j_i} & =0
\end{align}
for all $(i+1)$-tuples $0\leq j_0<\ldots <j_i\leq n$ (where, as usual, by $\hat{j}_k$ we indicate that we have omitted $j_k$ from the tuple $(j_0,\ldots,j_k,\ldots,j_i)$), we need to find $M\hookrightarrow N$, and a tuple $(x_{j_0\ldots j_{i-2}})_{0\leq j_0<\ldots < j_{i-2}\leq n}\in N^{\oplus \binom{n+1}{i-1}}$ such that 
\begin{align}\label{y=dx}
    y_{j_0 j_1\ldots j_{i-1}} & = \sum_{k=0}^{i-1} (-1)^k(\gamma_k-1) x_{j_0 \ldots \hat{j}_k j_{i-1}} 
\end{align}
for all $i$-tuples $0\leq j_0<\ldots<j_{i-1}\leq n$. By Lemma \eqref{(phi - 1) M}, after enlarging $M$, we may define the tuple $(x_{j_0\ldots j_{i-2}})_{0\leq j_0<\ldots < j_{i-2}\leq n}$ as follows. Set $x_{j_0\ldots j_{i-2}}:=0$ if $j_0=0$, otherwise set $x_{j_0\ldots j_{i-2}}$ to be any element in $N$ for which $(\varphi_q-1)(x_{j_0\ldots j_{i-2}})=y_{0j_0\ldots j_{i-2}}$. Then \eqref{y=dx} holds for all tuples $0\leq j_0<\ldots<j_{i-1}\leq n$ with $j_0=0$. Assume that we have chosen $(x_{j_0\ldots j_{i-2}})_{0\leq j_0<\ldots < j_{i-2}\leq n}$ so that \eqref{y=dx} is true whenever $0\leq j_0< j_0'$. It suffices to show that we may modify the tuple $(x_{j_0\ldots j_{i-2}})_{0\leq j_0<\ldots < j_{i-2}\leq n}$ further so that \eqref{y=dx} in fact holds true for all $0\leq j_0\leq j_0'$. Indeed, by the inductive hypothesis, and our assumption \eqref{dy=0} (applied to the $(i+1)$-tuples $0\leq j_0<j_0'<\ldots<j_{i-1}\leq n$ with $0\leq j_0< j_0'$), we obtain $(\gamma_{j_0}-1)z_{j_0' j_1\ldots j_{i-1}}=0$ for all $0\leq j_0<j_0'$, or equivalently, $z_{j_0' j_1\ldots j_{i-1}}\in M^{\varphi_q=1,\gamma_1=\ldots=\gamma_{j_0'-1}=1}$, where
\begin{align*}
    z_{j_0' j_1\ldots j_{i-1}}:=y_{j_0' j_1\ldots j_{i-1}} - \sum_{k=0}^{i-1} (-1)^k(\gamma_k-1) x_{j_0' \ldots \hat{j}_k j_{i-1}}.
\end{align*}
By Lemma \eqref{M phi=1 lemma} (applied to $i=j_0'-1\leq n-2$), we may thus enlarge $N$ further so that $z_{j_0' j_1\ldots j_{i-1}}=(\gamma_{j_0'}-1)z_{j_1\ldots j_{i-1}}$ for some $z_{j_1\ldots j_{i-1}}\in M^{\varphi_q=1,\gamma_1=\ldots=\gamma_{j_0'-1}=1}$. By replacing $x_{j_1\ldots j_{i-1}}$ with $x_{j_1\ldots j_{i-1}}+z_{j_1\ldots j_{i-1}}$, we see that \eqref{y=dx} now holds true for all $0\leq j_0\leq j_0'$, as desired. 
\end{proof}
\begin{lem}\label{M phi=1 lemma}
Given any $M\in \M_{\varphi,\Gamma}^{\mathrm{\acute{e}t},\lt}$, and any $x\in M^{\varphi_q=1,\gamma_1=\ldots=\gamma_{i}=1}$ with $0\leq i\leq n-2$, we can find an embedding $M\hookrightarrow N$ in $\ltet$ so that $x\in (\gamma_n-1)(N^{\varphi_q=1,\gamma_1=\ldots=\gamma_{i}=1}).$
\end{lem}
\begin{proof}
We will use descending induction on $0\leq i\leq n-2$. Assume first that $i=n-2$. We can pick $r\geq 0$ large enough so that $(\gamma_{n-1}-1)^r(x)=0$. To see this, let $U$ be an $\varphi_q$-stable open subgroup of $M$: such $U$ exists by Lemma \eqref{phi stable lattice phi-1 bijective} below. As the map $\gamma_{n-1}-1: M\to M$ is topologically nilpotent, we have $(\gamma_{n-1}-1)^r(x)\in U$ for all $r\geq 0$ sufficiently large. But this in fact forces $(\gamma_{n-1}-1)^r(x)=0$ by injectivity of the map $1-\varphi_q: U\to U$.

We claim that for each $0\leq k\leq r$, we can choose $N$ so that $(\gamma_{n-1}-1)^k(x)\in (\gamma_n-1)(N^{\varphi_q=1,\gamma_1=\ldots=\gamma_{n-2}=1})$. This is clear for $k=r$ by our choice of $r$. Assume the claim is true for $k+1$ (with $0\leq k\leq r-1$). Then we can choose $x'\in N^{\varphi_q=1,\gamma_1=\ldots=\gamma_{n-2}=1}$ so that $(\gamma_{n-1}-1)^{k+1}(x)=(\gamma_n-1)(x')$. Now applying the construction in the case of $H^1$ for the tuple $(0,\ldots,0,x',(\gamma_{n-1}-1)^k(x))$, we can enlarge $N$ further so that $(\gamma_{n-1}-1)^k(x)\in (\gamma_n-1)(N^{\varphi_q=1,\gamma_1=\ldots\gamma_{n-2}=1})$. In particular, we obtain the result for $i=n-2$ by setting $k=0$.

Assume that $0\leq i\leq n-3$ and that the result have been proved for all $i+1\leq j\leq n-2$. Again, we can pick $r\geq 0$ large enough so that $(\gamma_{i+1}-1)^{a_{i+1}}\ldots (\gamma_{n-1}-1)^{a_{n-1}}(x)=0$ for all nonnegative integers $a_{i+1},\ldots,a_{n-1}$ with sum $a_{i+1}+\ldots+a_{n-1}=r$. In particular, we trivially have $(\gamma_{i+1}-1)^{a_{i+1}}\ldots (\gamma_{n-1}-1)^{a_{n-1}}(x)\in (\gamma_n-1)M^{\varphi_q=1,\gamma_1=\ldots=\gamma_i=1}$. If $r\geq 1$, then we claim that the same property holds for $r-1$ after possibly enlarging $M$, i.e. that we can choose $N$ so that $(\gamma_{i+1}-1)^{a_{i+1}'}\ldots (\gamma_{n-1}-1)^{a_{n-1}'}(x)\in (\gamma_n-1)N^{\varphi_q=1,\gamma_1=\ldots=\gamma_i=1}$ for all tuples $(a_{i+1}',\ldots,a_{n-1}')$ of nonnegative integers of sum $r-1$. Indeed, for such tuple, we can write 
\begin{align*}
    \begin{cases}
    (\gamma_{i+1}-1)^{1+a_{i+1}'}\ldots (\gamma_{n-1}-1)^{a_{n-1}'}(x) & =(\gamma_n-1)(x_{i+1}) \\
    &\ldots\\
    (\gamma_{i+1}-1)^{a_{i+1}'}\ldots (\gamma_{n-1}-1)^{1+a_{n-1}'}(x) & =(\gamma_n-1)(x_{n-1})
    \end{cases}
\end{align*}
for some $x_{i+1},\ldots,x_{n-1}\in N^{\varphi_q=1,\gamma_1=\ldots=\gamma_i=1}$. In particular, we have $x_{jk}:=(\gamma_j-1)x_k-(\gamma_k-1)x_j\in M^{\varphi_q=1,\gamma_1=\ldots=\gamma_i=\gamma_n=1}$ for every $i+1\leq j<k\leq n-1$. 

We will now modify the elements $x_{i+1},\ldots,x_{n-1}$ suitably so that in fact $x_{jk}=0$ for all such $(j,k)$. More precisely, we will show by ascending induction on $i+1\leq j<n-1$ that we can modify the $x_k$ with $k>j$ so that $x_{jk}=0$ for all such $k$. Assume first that $j=i+1$. Using the inductive hypothesis for the case $i+1\leq n-2$, after possibly enlarging $M$, we may choose for each $j<k\leq n-1$ an element $y_{k}\in M^{\varphi_q=1,\gamma_1=\ldots=\gamma_{i}=\gamma_n=1}$ so that $x_{jk}=(\gamma_j-1)(y_k)$. We now simply replace each $x_k$ by $x_{k}-y_k$. Next assume that we have modified so that $x_{j'k}=0$ for all $i+1\leq j'<k\leq n-1$ with $j'<j$. In particular, as $x_{j'j}=0=x_{j'k}$, we have 
\begin{align*}
    (\gamma_{j'}-1)x_{jk} & = (\gamma_j-1) (\gamma_{j'}-1)x_k-(\gamma_{k}-1)(\gamma_{j'}-1)x_j\\
    &=(\gamma_j-1)(\gamma_k-1)x_{j'}-(\gamma_{k}-1)(\gamma_j-1)x_{j'}\\
    &=0.
\end{align*}
In other words, we in fact have $x_{jk}\in M^{\varphi_q=1,\gamma_1=\ldots=\gamma_{j-1}=\gamma_n=1}$. Thus, by applying the inductive hypothesis for $i+1\leq j\leq n-2$, we can find (after possibly further enlarging $N$), for each $k>j$ an element $z_k\in M^{\varphi_q=1,\gamma_1=\ldots=\gamma_{j-1}=\gamma_n=1}$ so that $x_{jk}=(\gamma_j-1)z_k$. Again, we can now replace each $x_k$ by $x_k-z_k$ so that in fact $x_{jk}=0$, as wanted (note that as we have designed so that $z_k$ are fixed by $\gamma_{j'}$ for all $j'<j$, this does not affect the modifcation that we made earlier on the $x_{j'k}$ with $j'<k$ (i.e. they are still zero)). 

The upshot is that now the tuple $(0,\ldots,0,x_{i+1},\ldots,x_{n-1},(\gamma_{i+1}-1)^{a_{i+1}'}\ldots(\gamma_{n-1}-1)^{a_{n-1}'}(x))$ is a 1-cocycle, and so by using the construction in the case of $H^1$, we can further enlarge $N$ so that $(\gamma_{i+1}-1)^{a_{i+1}'}\ldots(\gamma_{n-1}-1)^{a_{n-1}'}(x)\in (\gamma_n-1)N^{\varphi_q=1,\gamma_1=\ldots=\gamma_i=1}$ for all tuples $(a_{i+1}',\ldots,a_{n-1}')$ of nonnegative integers of sum $r-1$, as claimed. Again by continuing this procedure, we may arrive to the case $r=0$, where we clearly have $x\in (\gamma_n-1)N^{\varphi_q=1,\gamma_1=\ldots=\gamma_i=1}$, as wanted.
\end{proof}
\begin{lem}\label{(phi - 1) M}
Given any $M\in \M_{\varphi,\Gamma}^{\mathrm{\acute{e}t},\lt}$, and any $x\in M$, we can find an embedding $M\hookrightarrow N$ in $\ltet$ so that $x\in (\varphi_q-1)(N)$.
\end{lem}
\begin{proof}
It follows from Lemma \eqref{phi stable lattice phi-1 bijective} that $(\varphi_q-1)(M)$ is an open subgroup of $M$. We may thus pick $r\geq 0$ large enough so that $(\gamma_1-1)^{a_1}\ldots (\gamma_n-1)^{a_n}(x)\in (\varphi_q-1)(M)$ for all tuples $(a_1,\ldots,a_n)$ of nonnegative integers with sum $r$. As in the proof of Lemma \eqref{M phi=1 lemma}, we claim that if $r\geq 1$, then the same property holds for $r-1\geq 0$. Indeed, given any tuple $(a_1',\ldots,a_n')$ with sum $r-1$, we can write 
\begin{align*}
    \begin{cases}
    (\gamma_1-1)^{1+a_1'}\ldots (\gamma_n-1)^{a_n'}(x) & = (\varphi_q-1)(x_1)\\
    & \ldots \\
    (\gamma_1-1)^{a_1'}\ldots (\gamma_n-1)^{1+a_n'}(x) & = (\varphi_q-1)(x_n)
    \end{cases}
\end{align*}
for some $x_1,\ldots,x_n\in M$. Again, we have $x_{jk}:=(\gamma_j-1)x_k-(\gamma_k-1)x_j\in M^{\varphi_q=1}$ for all $1\leq j<k\leq n$. By applying Lemma \eqref{M phi=1 lemma} and its proof, we may modify the elements $x_1,\ldots,x_n$ inductively so that in fact $x_{jk}=0$ for all such $j<k$. Indeed, assume first that $j=1$. By the case $i=0$ of the that lemma, we can find $N$ so that $x_{1k}=(\gamma_1-1)y_k$ for some $y_k\in N^{\varphi_q=1}$. Then by replacing each $x_k$ with $x_k-y_k$, we obtain $x_{1k}=0$ for all $k>1$. Assume we have $x_{j'k}=0$ for all $1\leq j'<k\leq n$ with $j'<j$. Arguing as the proof of \textit{loc. cit.}, we see that in fact $x_{jk}\in M^{\varphi_q=1,\gamma_1=\ldots=\gamma_{j-1}=1}$. Then by the case $i=j-1\leq n-2$ of \textit{loc. cit.}, we may choose $z_k\in M^{\varphi_q=1,\gamma_1=\ldots=\gamma_{j-1}=1}$ so that $x_{jk}=(\gamma_j-1)(z_k)$. Again by replacing each $x_k$ with $x_k-z_k$, we may then assume that $x_{jk}=0=x_{j'k}$ for all $k>j>j'$, as desired.

Now by using the construction in the case of $H^1$ for the 1-cocycle $((\gamma_1-1)^{a_1'}\ldots (\gamma_n-1)^{a_n'}(x),x_1,\ldots,x_n)$, we may further enlarge $N$ so that $(\gamma_1-1)^{a_1'}\ldots (\gamma_n-1)^{a_n'}(x)\in (\varphi_q-1)(N)$, as claimed. Again, in the case $r=0$, we obtain $x\in (\varphi_q-1)(N)$, as desired.
\end{proof}
\begin{lem}\label{phi stable lattice phi-1 bijective}
Let $A$ be a Noetherian $\O/\varpi^a$-algebra for some $a\geq 1$, and let $M$ be a finite (not necessarily \' etale) $\varphi_q$-module over $\aka$. Then $M$ admits an open $\varphi_q$-stable subgroup $U$ on which the map $\varphi_q-1$ is bijective. 
\end{lem}
\begin{proof}
By considering $M$ as a finite \' etale $\varphi_q$-module over the subring $\mathbf{A}_{F,A}\subseteq \aka$, we may assume that $K=F$ (note that the canonical topology on $M$ does not depend on whether we view $M$ as a finite module over $\aka$ or $\mathbf{A}_{F,A}$). In this case the subring $\mathbf{A}_{F,A}^+$ of $\mathbf{A}_{F,A}$ is $\varphi_q$-stable, and we can in fact choose $U$ to be a $\varphi_q$-stable lattice in $M$. To see this, pick any lattice $\fM$ in $M$, and pick $n>0$ large enough so that $\Phi_M(\varphi_q^*\fM)\subseteq T^{-n}\fM$. As $\pi^a=0$ in $A$ and $\varphi_q(T)\equiv T^q\bmod{\pi \mathbf{A}_{F,A}^+}$, it follows easily from the binomial theorem that $\varphi_q(T^{M+a-1})$ is divisible by $T^{Mq}$ in $\mathbf{A}_{F,A}^+$. In particular, we have $\Phi_M(\varphi_q^*(T^{M+a-1}\fM))=\varphi_q(T^{M+a-1})\Phi_M(\varphi_q^*\fM)\subseteq T^{Mq-n}\fM$. If we pick $M$ large enough so that $Mq-n\geq M+a-1+q^{a-1}$, then $\fM':=T^{M+a-1}\fM$ is a $\varphi_q$-stable lattice in $M$ satisfying $\varphi_q(\fM')\subseteq T^{q^{a-1}}\fM'$. As $\varphi_q(T)\equiv T^q\bmod{\pi}$, we have $\varphi_q(T^{q^{a-1}})\equiv T^{q^a}\bmod{\pi^a}$. Thus for any $t\geq 0$, we have $\varphi_q(T^{tq^{a-1}}\fM')\subseteq T^{(1+tq)q^{a-1}}\fM'$. As $\fM'$ is $T$-adically complete (being finite over the Noetherian $T$-adically complete ring $\mathbf{A}_{F,A}$), it follows that for each $x\in \fM'$, the series $\sum_{l\geq 0}\varphi_q^l(x)$ converges in $\fM'$. We deduce that the map $1-\varphi_q: \fM'\to \fM'$ is bijective with inverse given by $\sum_{l\geq 0}\varphi_q^l$, as desired.
\end{proof}
\begin{proof}[Proof of Proposition \eqref{Herr cohomology and ext}]
This follows from Lemmas \eqref{effaceable ext} and \eqref{effaceable lubin tate}, uniqueness of universal $\delta$-functors, and the fact that we have a natural identification between $H^0(\C^\bullet(M))=M^{\varphi_q=1,\Gamma_K=1}$ and $\mathrm{Hom}_{\ltet}(\aka,M)=\ext_{\lt}^0(\aka,M)$.
\end{proof}
\begin{cor}\label{equivalence LT EG cohomology}
Assume $A$ is a finite type $\O/\varpi^a$-algebra for some $a\geq 1$. Let $M$ be a finite projective \' etale $(\varphi_q,\Gamma_K)$-module with $A$-coefficients, and let $M^{\eg}$ denote the $(\varphi,\Gamma_K^{\eg})$-module corresponding to $M$ under the equivalence in Proposition \eqref{equivalence LT EG}. Then $H^i(\C^\bullet(M))\cong H^i_{\eg}(\C^\bullet(M^{\eg}))$ for each integer $i$.  
\end{cor}
\begin{proof}
As the equivalence in \textit{loc. cit.} is exact and clearly takes $\aka$ to $\aka^{\eg}$, we see that $\ext_{\lt}^i(\aka,M)\cong \ext_{\eg}^i(\aka^{\eg},M^{\eg})$ for each $i\in \mathbf{Z}$. The result now follows from Proposition \eqref{Herr cohomology and ext} (applied to both $M$ and $M^{\eg}$).
\end{proof}
\begin{thm}\label{Herr perfect}
Let $A$ is a finite type $\O/\varpi^a$-algebra for some $a\geq 1$, and let $M$ be a finite projective \' etale $(\varphi_q,\Gamma_K)$-module with $A$-coefficients.
\begin{itemize}
    \item[\emph{(1)}] The Herr complex $\C^\bullet(M)$ is a perfect complex of $A$-modules, with tor-amplitude in $[0,2]$.
    \item[\emph{(2)}] If $B$ is a finite type $A$-algebra, then there is natural isomorphism in the derived category
    \begin{displaymath}
    \C^\bullet(M)\otimes_A^{\mathbf{L}} B\xrightarrow{\sim} \C^\bullet(M\otimes_{\aka}\ak{B}).
    \end{displaymath}
    In particular, there is a natural isomorphism 
    \begin{displaymath}
    H^2(C^\bullet(M))\otimes_AB\xrightarrow{\sim} H^2(\C^\bullet(M\otimes_{\aka}\ak{B}).
    \end{displaymath}
\end{itemize}
\end{thm}
\begin{proof}
(1) As $\C^\bullet(M)$ is a bounded complex of flat $A$-modules, showing $\C^\bullet(M)$ is perfect amounts to showing that its cohomology groups are finite $A$-modules. This follows from Corollary \eqref{equivalence LT EG cohomology} and the corresponding result in \cite[Thm. 5.1.22 (1)]{EG22} for the cyclotomic Herr complex. 

Having shown that the complex $\C^\bullet(M)$ is perfect, we can now prove part (2) by following exactly the proof of \cite[Thm. 5.1.22 (2)]{EG22} (which in turn follows the argument in \cite[Thm. 4.4.3]{KPX14}).

It remains to show that $\C^\bullet(M)$ has tor-amplitude in $[0,2]$. For this we need to check that $\C^\bullet(M)\otimes^{\mathbf{L}}_A N$ has cohomological amplitude in $[0,2]$ for all $A$-modules $N$. As in the proof of \cite[Cor. 5.1.25]{EG22}, it suffices to treat the case where $N=B$ for some finite $A$-algebra $B$. This follows from the isomorphism in part (2) and the fact that $H^i(\C^\bullet(M_B))\cong H^i_{\eg}(\C^\bullet(M_B^{\eg}))=0$ for all $i\notin [0,2]$, where $M_B:=M\otimes_{\aka}\ak{B}$.
\end{proof}
\begin{remark}
1. For showing that the Herr complex is perfect, it suffices by \cite[\href{https://stacks.math.columbia.edu/tag/07LU}{Tag 07LU}]{Sta21} to consider the case where $A$ is a finite type $\mathbf{F}$-algebra. Thus, in view of the above proof, it is in fact enough to invoke the comparison \eqref{equivalence LT EG} only for such algebras; in other words, we need the isomorphism $\X_{K,d}^{\lt}\xrightarrow{\sim} \X_{K,d}^{\eg}$ only on special fibers. On the other hand, it would be desirable to have a proof purely in the world of Lubin--Tate $(\varphi_q,\Gamma_K)$-modules.

2. Our first approach on showing perfectness of the Herr complex was to proceed along the lines of \cite[\textsection 5.1]{EG22} by using a kind of $\psi$-operator on Lubin--Tate $(\varphi_q,\Gamma_K)$-modules. However, this turned out to be not possible (as far as we know) since there is no $\Gamma_K$-equivariant left inverse of $\varphi_q: M\to M$ if $F\ne \mathbf{Q}_p$, cf. \cite[Prop. 3.2.6]{bergerpsi}.
\end{remark}
\begin{cor}\label{perfect complex represented 0 1 2}
Assume $A$ is a finite type $\O/\varpi^a$-algebra for some $a\geq 1$. Then the Herr complex $\C^\bullet(M)$ can be represented by a complex $[C^0\to C^1\to C^2]$ of finite locally free A-modules in degrees $[0,2]$.
\end{cor}
\begin{proof}
This follows from Theorem \eqref{Herr perfect}, and \cite[\href{https://stacks.math.columbia.edu/tag/0658}{Tag 0658}]{Sta21}.
\end{proof}
\subsubsection{Galois cohomology via the Herr complex}
As another consequence of the effaceablity of the functors $M\mapsto H^i(\C^\bullet(M)), i\geq 1$ (Lemma \eqref{effaceable lubin tate}), we recover the following comparison result. The approach employed here is similar to that in \cite{Her98}; see also \cite{HerrBerger} for a different approach.
\begin{thm}\label{comparison Galois coho}
Let $A$ be a complete local Noetherian $\O$-algebra with finite residue field of characteristic $p$, and let $V$ be a finite $A$-module equipped with a continuous linear representation of $G_K$. Then there are isomorphism of $A$-modules
\begin{align*}
    H^i(G_K,V)\xrightarrow{\sim} H^i(\C^\bullet(D_A(V))).
\end{align*}
which are functorial in $A$.
\end{thm}
\begin{proof}
First assume that $\m_A^n V=0$ for some $n$. The functor $D_A$ extends formally to an equivalence between the category $\mathrm{Rep}_A^{\mathrm{tor}}(G_K)$ of $\m_A$-power torsion linear $G_K$-representations, and the category $\M_{\varphi,\Gamma}^{\mathrm{Ind-\acute{e}t,tor}}$ of direct colimits of finite $\m_A$-power torsion \' etale $(\varphi,\Gamma)$-modules over $\aka$ (the point being that, unlike its subcategory of \textit{finite} $G_K$-representations, $\mathrm{Rep}_A^{\mathrm{tor}}(G_K)$ has enough injective objects). Clearly, if $M$ is an object in $\M_{\varphi,\Gamma}^{\mathrm{Ind-\acute{e}t,tor}}$, then $M$ has a natural $(\varphi_q,\Gamma_K)$-module structure, and we can define the Herr complex of $M$ exactly as for finite \' etale $(\varphi_q,\Gamma_K)$-modules. We can check easily that the arguments in Lemmas \eqref{effaceable lubin tate}, \eqref{M phi=1 lemma} and \eqref{(phi - 1) M} still make sense, and thus show that the functors $M\mapsto H^i(\C^\bullet(M)), i\geq 1$ remain effaceable on the category $\M_{\varphi,\Gamma}^{\mathrm{Ind-\acute{e}t,tor}}$. Now as the cohomological $\delta$-functor $(V\mapsto H^i(G_K,V))_{i\geq 0}$ is universal on $\mathrm{Rep}_{A}^{\mathrm{tor}}(G_K)$, we obtain the result in this case. 

For a general $V$, consider the inverse system $(\C^\bullet(G_K,V/\m_A^n V))_{n}$ of cochain complexes. As the transition maps are surjective, its derived limit can be computed by taking inverse limits termwise, i.e. 
\begin{displaymath}
R\lim \C^\bullet(G_K,V/\m_A^n V) =\varprojlim \C^\bullet(G_K,V/\m_A^n V)=C^\bullet_{\mathrm{cts}}(G_K,V).
\end{displaymath}
In particular, there is a Milnor short exact sequence 
\begin{displaymath}
0\to R^1\lim H^{i-1}(G_K,V/\m_A^n V)\to H^i(G_K,V)\to \varprojlim H^i(G_K,V/\m_A^n V)\to 0
\end{displaymath}
for each $i\in \mathbf{Z}$. By what we have just seen, each $H^{i-1}(G_K,V/\m_A^n V)$ is computed by $H^{i-1}(\C^\bullet(D_A(V/\m_A^n V))$, and hence finite by Theorem \eqref{Herr perfect}. It follows that the $R^1 \lim$ on the left vanishes, and we obtain 
\begin{displaymath}
H^i(G_K,V)=\varprojlim H^i(G_K,V/\m_A^n V).
\end{displaymath}
Arguing similarly for the inverse system $(\C^\bullet(D_A(V/\m_A^n V)))_n$ (and recalling that $D_A(V)=\varprojlim D_A(V/\m_A^n V)$), we obtain 
\begin{displaymath}
H^i(\C^\bullet(D_A(V)))=\varprojlim H^i(\C^\bullet(D_A(V/\m_A^n V))),
\end{displaymath}
hence the desired result.
\end{proof}
\subsubsection{Obstruction theory}
As in \cite[\textsection 5.1.33]{EG22}, one can use perfectness of the Herr complex to show that the stack $\X_{K,d}^{\lt}$ admits a nice obstruction theory in the sense of \cite[Defn. A.34]{EG22}.

Let $A$ be a finite type $\O/\varpi^a$-algebra for some $a\geq 1$, and let $M$ be a rank $d$ projective \' etale $(\varphi_q,\Gamma_K)$-module with $A$-coefficients, classified by a map $x: \Sp A\to \X_{K,d}^{\lt}$. Given a square-zero extension
\begin{displaymath}
0\to I\to A'\to A\to 0
\end{displaymath}
in which $A'$ is a finite type $\O/\varpi^a$-algebra, we want to consider the problem of deforming $M$ to a projective \' etale $(\varphi_q,\Gamma_K)$-module $M'$ with $A'$-coefficients. More formally, let $\mathrm{Lift}(x,A')$ be the set of isomorphism classes of pairs $(M',\iota)$ where $M'$ is a projective \' etale $(\varphi_q,\Gamma_K)$-module $M'$ with $A'$-coefficients, and $\iota$ is an isomorphism $M'\otimes_{A'}A\xrightarrow{\sim} M$ of $(\varphi_q,\Gamma_K)$-module with $A$-coefficients. 

First, we can lift $M$ uniquely to a rank $d$ projective $\mathbf{A}_{K,A'}$-module $\widetilde{M}$ (as a finite projective module always deforms uniquely along a nilpotent thickening). By viewing $\varphi: M\to M$ as a linear map $\varphi_q^*M\to M$, we see that $\varphi$ also lifts to a semilinear map $\widetilde{\varphi}: \widetilde{M}\to \widetilde{M}$. Note that the linearization $\Phi_{\widetilde{\varphi}}$ of $\widetilde{\varphi}$ is then automatically an isomorphism as it is a surjective map between finitely projective modules of the same rank (surjectivity follows from Nakayama's lemma and the fact that $\Phi_{\widetilde{\varphi}}$ is surjective modulo a nilpotent ideal). Similarly, we can also lift each $\gamma_i$ to a bijective semilinear map $\widetilde{\gamma}_i: \widetilde{M}\to \widetilde{M}$. However, the problem is that the actions of $\widetilde{\varphi}$ and $\widetilde{\gamma}_i$ on $\widetilde{M}$ may not commute with each other (the set $\mathrm{Lift}(x,A')$ is therefore empty in this case). On the other hand, as the following lemma shows, this is the only obstruction to lift $M$ to an $A'$-point of $\X_{K,d}^{\lt}$.
\begin{lem}\label{obstruction auto continuous}
Assume there exist pairwise commuting lifts $\widetilde{\varphi}, \widetilde{\gamma}_1,\ldots,\widetilde{\gamma}_n$ of $\varphi,\gamma_1,\ldots,\gamma_n$ respectively to $\widetilde{M}$. Then the induced semilinear action of $\Gamma_{K,\disc}$ on $\widetilde{M}$ is continuous. In other words, these lifts makes $\widetilde{M}$ into an \' etale $(\varphi_q,\Gamma_K)$-module with $A'$-coefficients which lifts $M$.  
\end{lem}
\begin{proof}
By assumption, $\widetilde{M}$ is an object of $\R_{K,d}^{\Gamma_{K,\disc}}(A')$, and we want to show that $\widetilde{M}$ in fact comes from an object of $\X_{K,d}^{\lt}\hookrightarrow \R_{K,d}^{\Gamma_{K,\disc}}$. Using Lemma \eqref{cartesian diagram xk rk xk basic}, we may reduce to the case $K$ is $F$-basic. In particular, we can apply the material on $T$-quasi-linear endomorphisms. More precisely, by Lemma \eqref{continuous action Gamma disc}, it suffices to show that the action of $\gamma_i-1$ on $\widetilde{M}$ is topologically nilpotent for each $1\leq i\leq [K:\mathbf{Q}_p]$. Arguing in as the proof of \cite[Lem. D.31]{EG22}, we may reduce to the case when $\widetilde{M}$ is free. In particular, we may choose a \textit{free} latice $\widetilde{\fM}$ in $\widetilde{M}$. Then $\fM:=\widetilde{\fM}\otimes_{A'}A$ is also a lattice in $M:=\widetilde{M}\otimes_{A'}A$. As the action of $\Gamma_{K,\disc}$ on $M$ is continuous by assumption, we can use Lemma \eqref{continuous action Gamma disc} again to find $m\gg 0$ large enough so that $(\gamma_i-1)^{m}\fM\subseteq T\fM$. Equivalently, we have $(\gamma_i-1)^{m}\widetilde{\fM}\subseteq T\widetilde{\fM}+I\widetilde{\fM}$. Repeating the same argument for the \textit{free} lattice $T\widetilde{\fM}$ in $\widetilde{\fM}$, we can choose some $m'\geq m$ large enough so that $(\gamma_i-1)^{m'}T\widetilde{\fM}\subseteq T^2\widetilde{\fM}+I(T\widetilde{\fM})$. As $I^2=0$, it follows that $(\gamma_i-1)^{m+m'}\widetilde{\fM}\subseteq T^2\widetilde{\fM}+I(T\widetilde{\fM})+I(T\widetilde{\fM}+I\widetilde{\fM})=T^2\widetilde{\fM}+IT\widetilde{\fM}\subseteq T\widetilde{\fM}$, as desired.  
\end{proof}
Let $\mathrm{ad}\;M:=\mathrm{Hom}_{\aka}(M,M)$ be the adjoint of $M$, endowed with its natural structure of an \' etale $(\varphi_q,\Gamma_K)$-module with $A$-coefficients. 

We are now ready to measure the obstruction for lifting $M$ to $A'$.
\begin{lem}\label{H2 lift}
There is a functorial obstruction element $o_x(A')\in H^2(\C^\bullet(\mathrm{ad}\;M)\otimes^{\mathbf{L}}_AI)$, which vanishes precisely when $x$ can be lifted to $\Sp A'$.
\end{lem}
\begin{proof}
As above, we can choose semilinear lifts $\widetilde{\gamma}_i$ of the $\gamma_i$ to $\widetilde{M}$ (again, we will denote $\varphi$ by $\gamma_0$ for ease of notation). We first check that the tuple
\begin{align}\label{tuple}
(\widetilde{\gamma}_i\widetilde{\gamma}_j\widetilde{\gamma}_i^{-1}\widetilde{\gamma}_j^{-1}-1)_{0\leq i<j\leq n}\in \C^2(\mathrm{ad}\;M\otimes_AI)
\end{align}
is a cocyle, i.e. that $(\gamma_k-1)(\widetilde{\gamma}_i\widetilde{\gamma}_j\widetilde{\gamma}_i^{-1}\widetilde{\gamma}_j^{-1}-1)-(\gamma_j-1)(\widetilde{\gamma}_i\widetilde{\gamma}_k\widetilde{\gamma}_i^{-1}\widetilde{\gamma}_k^{-1}-1)+(\gamma_i-1)(\widetilde{\gamma}_j\widetilde{\gamma}_k\widetilde{\gamma}_j^{-1}\widetilde{\gamma}_k^{-1}-1)=0$ in $\mathrm{ad}\;M\otimes_AI$ for all $0\leq i<j<k\leq n$. It is convenient to consider this as an equation in $\mathrm{ad}_{\mathbf{A}_{K,A'}}\widetilde{M}\supseteq \mathrm{ad}\;M\otimes_AI$; in particular, for each $0\leq i\leq n$, the action of $\gamma_i$ on $\mathrm{ad}\;M\otimes_AI$ is now nothing but the conjugation by $\widetilde{\gamma}_i: \widetilde{M}\to \widetilde{M}$. Now fix $0\leq i<j<k\leq n$, and write $Y:=\widetilde{\gamma}_i\widetilde{\gamma}_k\widetilde{\gamma}_i^{-1}\widetilde{\gamma}_k^{-1}-1, Z:=\widetilde{\gamma}_i\widetilde{\gamma}_j\widetilde{\gamma}_i^{-1}\widetilde{\gamma}_j^{-1}-1$. Then $\widetilde{\gamma}_i\widetilde{\gamma}_k\widetilde{\gamma}_i^{-1}=(1+Y)\widetilde{\gamma}_k$ and $\widetilde{\gamma}_i\widetilde{\gamma}_j\widetilde{\gamma}_i^{-1}=(1+Z)\widetilde{\gamma}_j$. Multiplying these equations gives $\widetilde{\gamma}_i\widetilde{\gamma}_k\widetilde{\gamma}_j\widetilde{\gamma}_i^{-1}=\widetilde{\gamma}_k\widetilde{\gamma}_j+\widetilde{\gamma}_kZ\widetilde{\gamma}_j+Y\widetilde{\gamma}_k\widetilde{\gamma}_j$ whence $\widetilde{\gamma}_kZ\widetilde{\gamma}_k^{-1}+Y=\widetilde{\gamma}_i\widetilde{\gamma}_k\widetilde{\gamma}_j\widetilde{\gamma}_i^{-1}\widetilde{\gamma}_j^{-1}\widetilde{\gamma}_k^{-1}-1$. Similarly, we have $\widetilde{\gamma}_jY\widetilde{\gamma}_j^{-1}+Z=\widetilde{\gamma}_i\widetilde{\gamma}_j\widetilde{\gamma}_k\widetilde{\gamma}_i^{-1}\widetilde{\gamma}_k^{-1}\widetilde{\gamma}_j^{-1}-1$. It follows that $(\gamma_k-1)Z-(\gamma_i-1)Y=\widetilde{\gamma}_i\widetilde{\gamma}_k\widetilde{\gamma}_j\widetilde{\gamma}_i^{-1}\widetilde{\gamma}_j^{-1}\widetilde{\gamma}_k^{-1}-\widetilde{\gamma}_i\widetilde{\gamma}_j\widetilde{\gamma}_k\widetilde{\gamma}_i^{-1}\widetilde{\gamma}_k^{-1}\widetilde{\gamma}_j^{-1}$. We are thus reduced to showing that $\widetilde{\gamma}_i\widetilde{\gamma}_k\widetilde{\gamma}_j\widetilde{\gamma}_i^{-1}\widetilde{\gamma}_j^{-1}\widetilde{\gamma}_k^{-1}-\widetilde{\gamma}_i\widetilde{\gamma}_j\widetilde{\gamma}_k\widetilde{\gamma}_i^{-1}\widetilde{\gamma}_k^{-1}\widetilde{\gamma}_j^{-1}+\widetilde{\gamma}_i\widetilde{\gamma}_j\widetilde{\gamma}_k\widetilde{\gamma}_j^{-1}\widetilde{\gamma}_k^{-1}\widetilde{\gamma}_i^{-1}-\widetilde{\gamma}_j\widetilde{\gamma}_k\widetilde{\gamma}_j^{-1}\widetilde{\gamma}_k^{-1}=0$. We can rewrite the left hand side as $\widetilde{\gamma}_i\widetilde{\gamma}_j\widetilde{\gamma}_k(\widetilde{\gamma}_j^{-1}\widetilde{\gamma}_k^{-1}\widetilde{\gamma}_i^{-1}-\widetilde{\gamma}_i^{-1}\widetilde{\gamma}_k^{-1}\widetilde{\gamma}_j^{-1})+(\widetilde{\gamma}_i\widetilde{\gamma}_k\widetilde{\gamma}_j\widetilde{\gamma}_i^{-1}-\widetilde{\gamma}_j\widetilde{\gamma}_k)(\widetilde{\gamma}_k\widetilde{\gamma}_j)^{-1}=\widetilde{\gamma}_i(\widetilde{\gamma}_j\widetilde{\gamma}_k-\widetilde{\gamma}_k\widetilde{\gamma}_j)(\widetilde{\gamma}_j^{-1}\widetilde{\gamma}_k^{-1}\widetilde{\gamma}_i^{-1}-\widetilde{\gamma}_i^{-1}\widetilde{\gamma}_k^{-1}\widetilde{\gamma}_j^{-1})+(\widetilde{\gamma}_i\widetilde{\gamma}_k\widetilde{\gamma}_j\widetilde{\gamma}_i^{-1}-\widetilde{\gamma}_j\widetilde{\gamma}_k)((\widetilde{\gamma}_k\widetilde{\gamma}_j)^{-1}-(\widetilde{\gamma}_j\widetilde{\gamma}_k)^{-1})=\widetilde{\gamma}_i\widetilde{\gamma}_j\widetilde{\gamma}_k(1-\widetilde{\gamma}_k^{-1}\widetilde{\gamma}_j^{-1}\widetilde{\gamma}_k\widetilde{\gamma}_j)(1-\widetilde{\gamma}_i^{-1}\widetilde{\gamma}_k^{-1}\widetilde{\gamma}_j^{-1}\widetilde{\gamma}_i\widetilde{\gamma}_k\widetilde{\gamma}_j)\widetilde{\gamma}_j^{-1}\widetilde{\gamma}_k^{-1}\widetilde{\gamma}_i^{-1}+\widetilde{\gamma}_j\widetilde{\gamma}_k(\widetilde{\gamma}_k^{-1}\widetilde{\gamma}_j^{-1}\widetilde{\gamma}_i\widetilde{\gamma}_k\widetilde{\gamma}_j\widetilde{\gamma}_i^{-1}-1)(1-\widetilde{\gamma}_k^{-1}\widetilde{\gamma}_j^{-1}\widetilde{\gamma}_k\widetilde{\gamma}_j)\widetilde{\gamma}_j^{-1}\widetilde{\gamma}_k^{-1}$. The last expression is now clearly zero as $I^2=0$. 

We can now define $o_x(A')$ to be the image in $H^2(\C^\bullet(\mathrm{ad}\;M\otimes_AI))$ of the tuple \eqref{tuple}. It remains to check that $o_x(A')$ is independent of the choice of the lifts $\widetilde{\gamma}_i$, and vanishes if and only if $\mathrm{Lift}(x,A')\ne \emptyset$. Indeed, any other choice of lifts $\widetilde{\gamma}_i'$ is of the form $\widetilde{\gamma}_i'=(1+X_i)\widetilde{\gamma}_i$ for some $X_i\in \mathrm{ad}\;M\otimes_AI$. From this, one can check easily that 
\begin{displaymath}
(\widetilde{\gamma}_i'\widetilde{\gamma}_j'(\widetilde{\gamma}_i')^{-1}(\widetilde{\gamma}_j')^{-1}-1)-(\widetilde{\gamma}_i\widetilde{\gamma}_j\widetilde{\gamma}_i^{-1}\widetilde{\gamma}_j^{-1}-1)=(\gamma_i-1)X_j-(\gamma_j-1)X_i.
\end{displaymath}
This shows that the class $o_x(A')\in H^2(\C^\bullet(\mathrm{ad}\;M\otimes_AI))$ is well-defined, and vanishes if and only if we can choose commuting lifts $\widetilde{\gamma}_i$. By Lemma \eqref{obstruction auto continuous}, the latter is equivalent to the condition that $\mathrm{Lift}(x,A')\ne \emptyset$, as desired (note also that as $\C^\bullet(\mathrm{ad}\;M)$ is a complex of flat $A$-modules, the derived tensor $\C^\bullet(\mathrm{ad}\;M)\otimes^{\mathbf{L}}_AI$ simplifies to $\C^\bullet(\mathrm{ad}\;M\otimes_AI)$). 
\end{proof}
\begin{lem}\label{H1 lift}
Let $F$ be a finitely generated $A$-module. Then there is a natural isomorphism $\mathrm{Lift}(x,A[F])\xrightarrow{\sim} H^1(\C^\bullet(\mathrm{ad}\;M)\otimes^{\mathbf{L}}_AF)$ of $A$-modules. 
\end{lem}
\begin{proof}
As before, a lift of $M$ to $A[F]$ is determined by lifts $\widetilde{\gamma}_i$ of the $\gamma_i$ on $M$. Given any such lifts, we can write $\widetilde{\gamma}_i=(1+X_i)\gamma_i$ for some $X_i\in \mathrm{ad}\;M\otimes_AF$ (note that we have abusively denoted also by $\gamma_i$ the lifts corresponding to the trivial lift $M\otimes_AA[F]$). It is easy to check that the lifts $\widetilde{\gamma}_i$ commute with each other (or equivalently, define an element of $\mathrm{Lift}(x,A[F])$) if and only if the tuple $(X_0,\ldots,X_n)$ lies $Z^1(\C^\bullet(\mathrm{ad}\;M\otimes_A F))$. Furthermore, as the endomorphisms of the trivial lifting $M\otimes_A A[F]$ are given by $1+X$ for $X\in\mathrm{ad}\;M\otimes_A F$, we see that the lift determined by a tuple $(X_0,\ldots,X_n)\in Z^1(\C^\bullet(\mathrm{ad}\;M\otimes_A F))$ is trivial if and only if there exists $X\in \mathrm{ad}\;M\otimes_AF$ such that $\widetilde{\gamma}_i=(1+X)\gamma (1+X)^{-1}$ for all $0\leq i\leq n$. It is easy to check that this is equivalent to saying that $(X_0,\ldots,X_n)=d^0(-X)$ is a 1-coboundary. Thus, we obtain a bijection $H^1(\C^\bullet(\mathrm{ad}\;M\otimes_AF))\xrightarrow{\sim} \mathrm{Lift}(x,A[F])$ at the level of sets. That it is also $A$-linear can be done exactly as in the proof of \cite[Lem. 5.1.35]{EG22}.
\end{proof}
\begin{cor}\label{nice obstruction theory}
$\X_{K,d}^{\lt}$ admits a nice obstruction theory in the sense of \cite[Defn. A.34]{EG22}.
\end{cor}
\begin{proof}
This follows by combining Theorem \eqref{Herr perfect}, Lemmas \eqref{H1 lift} and \eqref{H2 lift}.
\end{proof}
\subsection{Families of extensions}\label{families of extensions}
We now briefly indicate how to proceed along the lines of the arguments in \cite[Chap. 5]{EG22} to give an alternative proof that $\X_{K,d}^{\lt}$ is a Noetherian formal algebraic stack using only perfectness of the Lubin--Tate Herr complex (in particular, we are not using directly the comparison in Corollary \eqref{equivalence LT EG stacks} between $\X_{K,d}^{\lt}$ and the Emerton--Gee stack $\X_{K,d}^{\eg}$). See also Remark \eqref{equidimensional unsolved} below.

First, by perfectness of the Lubin--Tate Herr complex, we see that if $M$ is a finite projective \' etale $(\varphi_q,\Gamma_K)$-module with coefficients in some finite type $\O/\varpi^a$-algebra $A$, then the cohomology group $H^2(\C^\bullet(M))$ defines a coherent sheaf on $T=\Sp(A)$. In view of the relation with Galois cohomology (and in accordance with the notation in \cite{EG22}), we denote this sheaf by $H^2(G_K,\rho_T)$, where $\rho_T: T\to \X_{K,d}^{\lt}$ is the map classifying $M$. More generally, as the formation of $H^2$ is compatible with arbitrary finite type base change, by gluing we can extend the definition of $H^2(G_K,\rho_T)$ to an arbitrary (not necessarily affine) scheme $T$ of finite type over $\O/\varpi^a$ together with a family $\rho_T: T\to \X_{K,d}^{\lt}$.

Having also defined the universal unramified character in our setting (Definition \eqref{universal unramified}), it is then straightforward to check that the results in \cite[\textsection 5.3]{EG22} go over unchanged in our setting. In particular, we have an obvious analogue of \cite[Prop. 5.3.8]{EG22}, giving a description of families of extensions which are generally maximally nonsplit of niveau 1.

Next, by using Corollary \eqref{perfect complex represented 0 1 2} in place of \cite[Cor. 5.1.25]{EG22}, one can easily adapt that the construction of families of extensions in \cite[\textsection 5.4]{EG22} to our setting. More precisely, starting with a family $\overline{\rho}_T: T\to (\X_{d,\mathrm{red}}^{\lt})_{\overline{\mathbf{F}}_p}$ on a reduced affine $\overline{\mathbf{F}}_p$-scheme of finite type, and any representation $\overline{\alpha}: G_K\to \mathrm{GL}_a(\overline{\mathbf{F}}_p)$ for which $\mathrm{Ext}^2_{G_K}(\overline{\alpha},\overline{\rho}_t)$ is of constant rank for varying $t\in T(\overline{\mathbf{F}}_p)$, one can construct, for each fixed integer $n\geq 1$, a (geometric) vector bundle $V\to T$ together with a morphism 
\begin{displaymath}
V \to (\X_{d+an,\mathrm{red}}^{\lt})_{\overline{\mathbf{F}}_p}
\end{displaymath}
parametrizing a universal family of extensions 
\begin{displaymath}
0\to \overline{\rho}_T\otimes_{\O_T}\O_V\to \E_V\to \overline{\alpha}^{\oplus n}\otimes_{\overline{\mathbf{F}}_p}\O_V\to 0.
\end{displaymath}
Again, one can check that the results in \textit{loc. cit.} remain valid in the current context. One can then follow the proof of \cite[Thm. 5.5.12]{EG22} to obtain the following result. (For the algebraicity part, the rough idea is that by iterating the above construction, we obtain families of \' etale $(\varphi_q,\Gamma_K)$-modules parametrized by the various vector bundles $V$ appearing; as any Galois representation can be written as an iterated extension of irreducible representations, this will ultimately gives a cover of $\X_{d,\mathrm{red}}$. For obtaining the desired dimension, we will need to actually compute the dimensions of the various families of extensions arising from this construction, cf. \cite[Prop. 5.4.4]{EG22}.)
\begin{thm} $\X_{d,\mathrm{red}}^{\lt}$ is a finitely presented algebraic stack over $\mathbf{F}$ of dimension $[K:\mathbf{Q}_p]d(d-1)/2$.
\end{thm}
In particular, by combining algebraicity of $\X_{d,\mathrm{red}}$ with Corollary \eqref{nice obstruction theory}, we obtain another proof that $\X_{K,d}^{\lt}$ is a Noetherian algebraic stack (see the proof of \cite[Cor. 5.5.18]{EG22}, which ultimately relies on the criteria of \cite[Cor. 6.6 and Thm. 11.13]{Eme}).
\begin{remark}\label{equidimensional unsolved}
In fact, the inductive argument in \cite[Thm. 5.5.12]{EG22} also allows us to construct, for each Serre weight $\underline{k}$, an irreducible component $\X_{d,\mathrm{red},\overline{\mathbf{F}}_p}^{\underline{k},\lt}$ of $(\X_{d,\mathrm{red}}^{\lt})_{\overline{\mathbf{F}}_p}$ of dimension $[K:\mathbf{Q}_p]d(d-1)/2$, whose generic $\overline{\mathbf{F}}_p$-points are maximally nonsplit of niveau 1 and weight $\underline{k}$ in the sense of \cite[Defn. 5.5.1]{EG22}, and that the union of the $\X_{d,\mathrm{red},\overline{\mathbf{F}}_p}^{\underline{k},\lt}$, together with a closed substack of dimension strictly smaller than $[K:\mathbf{Q}_p]d(d-1)/2$, covers $(\X_{d,\mathrm{red}}^{\lt})_{\overline{\mathbf{F}}_p}$. In order to show that the $ \X_{d,\mathrm{red},\overline{\mathbf{F}}_p}^{\underline{k},\lt}$ exhaust the irreducible components of $(\X_{d,\mathrm{red}}^{\lt})_{\overline{\mathbf{F}}_p}$, we need to show that the latter is equidimensional. In the cyclotomic case, this is achieved by first constructing closed substacks of $\X_{d}^{\lt}$ corresponding to crystalline representations, computing their versal rings at finite type points, and then using the existence of crystallne lifts of mod $p$ representations. In order to adapt this strategy to the Lubin--Tate setting, it seems necessary to first have at one's disposal a description of $G_K$-stable $\O_F$-lattices in crystalline $F$-linear representations of $G_K$ in terms of Breuil--Kisin--Fargues modules over $\mathbf{A}_{\mathrm{inf},F}:=W_{\O_F}(\O_{\mathbf{C}}^\flat)$. It seems likely that these are also related to modifications of vector bundles on the Fargues--Fontaine curve $X_F$ associated to $F$ (and $\mathbf{C}^\flat$); the main difference with the case $F=\mathbf{Q}_p$ is that the relevant modifications will now take place at finitely many closed points in $X_F$ (namely those living over the distinguished closed point of $X_{\mathbf{Q}_p}$), rather than only at one single point. We hope to discuss this in more details in a subsequent work.
\end{remark}
\section{The rank one case}
We conclude by giving a simple description of the stack $\X_{K,d}^{\lt}$ in the case $d=1$. To this end, let $\xan$ be the functor on $\varpi$-adically complete $\O$-algebras taking $A$ to the set of (continuous) characters $\delta: W_K\to A^\times$. We equip $\xan$ with the trivial action of $\widehat{\mathbf{G}}_m$, the $\varpi$-adic completion of $\mathbf{G}_{m,\O}$.
\begin{prop}
There is an isomorphism 
\begin{displaymath}
\X_{K,1}^{\lt}\xrightarrow{\sim}  \left[\xan/\widehat{\mathbf{G}}_m\right].
\end{displaymath}
\end{prop}
\begin{proof}
Of course this follows at once from the comparison \eqref{equivalence LT EG stacks}, and the corresponding description for the stack $\X_{K,1}^{\eg}$ (\cite[Prop. 7.2.17]{EG22}). However, it is (again) enough to invoke only the comparison on special fibers. First, by using the universal unramified character \eqref{universal unramified}, we can construct a map $\xan\to \X_{K,1}^{\lt}, \delta\mapsto \aka(\delta)$ which recovers the equivalence \eqref{Dee equivalence} in case $A$ is a finite local Artinian $\O$-algebra, cf. Section 2 in \cite{Pham}. Again, this induces a monomorphism 
\begin{displaymath}
 \left[\xan/\widehat{\mathbf{G}}_m\right]\hookrightarrow \X_{K,1}^{\lt},
\end{displaymath}
which we want to show is an isomorphism. As in the proof of Theorem 3.1 in \textit{loc. cit.}, it in fact suffices to show that the map 
\begin{displaymath}
 [\mathbf{G}_m/\mathbf{G}_m]\hookrightarrow \overline{\X}_{K,1}^{\lt}
\end{displaymath}
(where the target denotes the special fiber of $\X_{K,1}^{\lt}$)
induced by the universal unramified character $\mathbf{G}_m\to \overline{\X}_{K,1}^{\lt}$ is a closed immersion. Thus, it suffices to show that the last map corresponds to the ``cyclotomic" universal unramified character $\mathbf{G}_m\to \overline{\X}_{K,1}^{\eg}$ in \textit{loc. cit.} via the isomorphism $\overline{\X}_{K,1}^{\lt}\xrightarrow{\sim} \overline{\X}_{K,1}^{\eg}$ of special fibers. This is straightforward.
\end{proof}
\begin{appendices}
  \titleformat{\section}{\centering\large\sc}{\appendixname{} \thesection.}{0.7em}{\centering\large\sc}
  \section{\texorpdfstring{$T$}{T}-quasi-linear endomorphisms}\label{T quasi linear}
We now briefly recall some material on $T$-quasi-linear endomorphisms, adapted to our slightly more general setting. We will be content with giving only the results that we need in our proof of Ind-algebraicity of $\X_{K,d}^{\lt}$; for a more detailed account, we refer the reader to Appendix D of \cite{EG22}.

Fix a finite extension $F/\mathbf{Q}_p$ with uniformizer $\pi$. Fix also a finite unramified extension $K/F$, and a finite extension $E/F$ with uniformizer $\varpi$ and ring of integers $\O$. If $A$ is $\varpi$-adically complete $\O$-algebra, we let $\mathbf{A}_A^+:=(\O_K\otimes_{\O_F}A)[[T]]$, and let $\mathbf{A}_A$ be the $p$-adic completion of $\mathbf{A}_A^+[1/T]$. As usual, finite projective modules over $\mathbf{A}_A$ or $\mathbf{A}_A^+$ will be endowed with their canonical topology (cf. \cite[Rem. D.2]{EG22}).
\begin{defn}
Let $A$ be an $\O/\varpi^a$-algebra for some $a\geq 1$, and let $M$ be a finite projective $\mathbf{A}_A$-module. A $T$-quasi-linear endomorphism of $M$ is a continuous $\O_K\otimes_{\O_F}A$-linear morphism $f: M\to M$ which furthermore satisfies: there exist $a(T)\in (\mathbf{A}_A^+)^\times$ and $b(T)\in (\pi,T)\mathbf{A}_A^+$ such that 
\begin{displaymath}
f(Tm)=a(T)Tf(m)+b(T)Tm
\end{displaymath}
for every $m\in M$.
\end{defn}
\begin{remark}
The above definition recovers \cite[Defn. D.17]{EG22} in case $F=\mathbf{Q}_p$.
\end{remark}
\begin{lem}\label{lemma D.18}
Assume $f$ is $T$-quasi-linear, then for all $n\in\mathbf{Z}$, we can write
\begin{displaymath}
f(T^n m)=a(T)^n T^n f(m)+ b_n(T) T^n m
\end{displaymath}
for all $m\in M$, where $a(T)\in (\mathbf{A}_A^+)^\times$ and $b_n(T)\in (\pi,T)\mathbf{A}_A^+$.
\end{lem}
\begin{proof}
See the proof of \cite[Lem. D.18]{EG22}.
\end{proof}
\begin{lem}\label{lemma D.19}
Let A be an $\O/\varpi^a$-algebra for some $a\geq 1$, and let $M$ be a
finite projective $\mathbf{A}_A$-module. Let $f$ be a $T$-quasi-linear endomorphism of $M$, and let $\fM$ be a lattice in $M$. Then there is an integer $m\geq 0$ such that $f^n(T^s\fM)\subseteq T^{s-mn}\fM$ for all $s\in\mathbf{Z}$ and $n\geq 0$.
\end{lem}
\begin{proof}
As $f$ is continuous \textit{by definition}, we can choose $m\geq 0$ large enough so that $f(T^m\fM)\subseteq \fM$. From this and Lemma \eqref{lemma D.18}, it is easy to see that $f(T^s\fM)\subseteq T^{s-m}\fM$ for all $s\in\mathbf{Z}$. The lemma now follows by an evident induction.
\end{proof}
\begin{lem}\label{topologically nilpotent}
Let A be an $\O/\varpi^a$-algebra for some $a\geq 1$. Let $M$ be a
finite projective $\mathbf{A}_A$-module, and let $f$ be a $T$-quasi-linear endomorphism of $M$. The following are equivalent:
\begin{itemize}
    \item[\emph{(1)}] $f$ is topologically nilpotent.
    \item[\emph{(2)}] There exists a lattice $\fM\subseteq M$ and some $n\geq 1$ such that $f^n(\fM)\subseteq (\pi,T)\fM$. 
    \item[\emph{(3)}] For any lattice $\fM\subseteq M$ and any $m\geq 1$, we have $f^n(\fM)\subseteq T^{m}\fM$ for all sufficiently large $n$.
\end{itemize}
\end{lem}
\begin{proof}
The proof is identical to that of \cite[Lem. D.21]{EG22}, except that we replace $p$ by $\pi$ everywhere, and appeal to Lemma \eqref{lemma D.18} in place of \cite[Lem. D.18]{EG22}.
\end{proof}
We now specialize to our main case of interest. To this end, assume that $\mathbf{A}_A$ is endowed with a continuous action of $\mathbf{Z}_p$ by $\O_K\otimes_{\O_F}A$-algebra automorphisms, which preserves $\mathbf{A}_A^+$, and moreover satisfies
\begin{align}\label{gamma T - T}
\gamma(T)-T\in (\pi,T)T\mathbf{A}_A^+
\end{align}
for some topological generator $\gamma$ of $\mathbf{Z}_p$.
\begin{lem}\label{gamma -1 quasi-linear}
Let $M$ be a finite projective $\mathbf{A}_A$-module which is endowed with a semilinear action of the subgroup $\langle\gamma\rangle$ of $\mathbf{Z}_p$, then $f:=\gamma^n-1$ is a $T$-quasi-linear endomorphism of $M$ for any $n\geq 1$.
\end{lem}
\begin{proof}
By induction, it is easy to see that $\gamma^n(T)-T\in (\pi,T)T\mathbf{A}_A^+$ for all $n\geq 1$. As $(\gamma^n-1)(Tm)=\gamma^n(T)(\gamma^n-1)(m)+(\gamma^n(T)-T)m$, we have $f(Tm)=a(T)^n T^n f(m)+b_n(T) Tm$ where $a(T):=\gamma(T)/T\in (\mathbf{A}_A^+)^\times$, and $b_n(T):=(\gamma^n(T)-T)/T\in (\pi,T)\mathbf{A}_A^+$. As we already require $\gamma$ to be $\O_K\otimes_{\O_F}A$-linear, it remains to check that $f: M\to M$ is continuous. To see this, let $\fM$ be an arbitrary lattice in $M$, say generated by $m_1,\ldots,m_r\in M$. Let $\mathfrak{N}$ be the sub-$\mathbf{A}_A^+$-module of $M$ generated by $\gamma^{-n}(m_1),\ldots,\gamma^{-n}(m_r)$. As $\gamma$ acts on $\mathbf{A}_A$ and $M$ by automorphisms, and moreover preserves $\mathbf{A}_A^+$, it is clear that $\mathfrak{N}$ is a lattice in $M$ verifying $\gamma^n(\mathfrak{N})\subseteq \fM$. Thus, $\gamma^n$ (hence $f$) is continuous, as desired. 
\end{proof}
\begin{remark}
We have seen in Lemma \eqref{3.2.18} that in case $K$ is $F$-basic, the action of $\gamma\in \Gamma_K$ on $\mathbf{A}_{K|F,A}$ indeed satisfies \eqref{gamma T - T}. Accordingly\footnote{We should also check that $\gamma$ fixes the subring $W_{\O_F}(k_{K,\infty})\otimes_{\O_F}A$, but this is clear because $k_{K,\infty}=k_K$ if $K$ is $F$-basic.}, we can apply Lemma \eqref{gamma -1 quasi-linear} in the case where ($K$ is $F$-basic and) $M$ is an object of $\R_{K,d}^{\Gamma_\disc}$, and the semilinear action of $\langle \gamma\rangle$ is given by restricting the action of $\Gamma_\disc$ on $M$. 
\end{remark}
\begin{lem}\label{continuous action Gamma disc}
Assume that ${A}$ is an $\O/\varpi^a$-algebra for some $a\geq 1$, and that $\mathbf{A}_A$ is endowed with an action of $\mathbf{Z}_p$ as above.  Let $M$ be a finite projective
$\mathbf{A}_A$-module, equipped with a semi-linear action of $\langle \gamma\rangle\subseteq \mathbf{Z}_p$. Then the following
are equivalent:
\begin{itemize}
    \item[\emph{(1)}] The action of $\langle\gamma \rangle$ extends (necessarily uniquely) to a continuous action of $\mathbf{Z}_p$. 
    \item[\emph{(2)}] The action of $\langle \gamma\rangle$ on $M$ is continuous. 
    \item[\emph{(3)}] For any lattice $\fM\subseteq M$ and any $n\geq 1$, there exists $s\geq 0$ such that $(\gamma^{p^s}-1)(\fM)\subseteq T^n\fM$.
    \item[\emph{(4)}] For some lattice $\fM\subseteq M$ and some $s\geq 0$, we have $(\gamma^{p^s}-1)(\fM)\subseteq T\fM$.
    \item[\emph{(5)}] The action of $\gamma-1$ on $M$ is topologically nilpotent. 
\end{itemize}
\end{lem}
\begin{proof}
The proof is essentially identical to that of \cite[Lem. D.28]{EG22}, but to make sure that everything works out as expected, we will provide the details. By \cite[Lem. D.15]{EG22}, (1) and (2) are equivalent. Clearly, $(3) \Longrightarrow (4)$. Now if (4) holds, then as $(\gamma^{p^s}-1)\equiv (\gamma-1)^{p^s}\bmod{p}$, we deduce that $(\gamma-1)^{p^s}(\fM)\subseteq (p,T)\fM\subseteq (\pi,T)\fM$, and so $\gamma-1$ is topologically nilpotent by Lemma \eqref{topologically nilpotent} (2). Conversely, assume (5) holds. By Lemma \eqref{topologically nilpotent} (3), we may choose $s\geq 0$ large enough so that $(\gamma-1)^{p^s}\fM\subseteq T^n\fM$. As $(\gamma-1)^{p^s}\equiv (\gamma^{p^s}-1)\bmod{p}$, we see that $(\gamma^{p^s}-1)\fM\subseteq (p,T^n)\fM$; in particular that (3) holds when $a=1$. Assume now that (3) is true up to $a-1\geq 1$. By the inductive hypothesis (applied to the lattice $(\fM+\varpi^{a-1}M)/\varpi^{a-1}M$ inside $M/\varpi^{a-1} M$), we can choose $s$ large enough so that $(\gamma^{p^s}-1)\fM\subseteq T^n\fM+\varpi^{a-1}M$. In particular, we have $p(\gamma^{p^s}-1)\fM\subseteq T^n\fM$. Now as $\gamma^{p^s}-1$ is topologically nilpotent, we can use Lemma \eqref{topologically nilpotent} again to pick $t$ large enough so that $(\gamma^{p^s}-1)^{p^t}\fM\subseteq T^n\fM$. As $(\gamma^{p^{s+t}}-1)=((\gamma^{p^s}-1)+1)^{p^t}-1\equiv (\gamma^{p^s}-1)^{p^t}\bmod{p(\gamma^{p^s}-1})$, it follows that $(\gamma^{p^{s+t}}-1)\fM\subseteq T^n\fM$, as desired. 

Assume now that (1) and (2) hold. Let $\fM\subseteq M$ be an arbitrary lattice, say with generators $m_1,\ldots,m_r$ of $\fM$. By \cite[Lem. D.13 (2)(a)]{EG22}, we can choose $s\geq 0$ sufficiently large so that $(\gamma^{p^s}-1)m_i\in T\fM$ for all $i$. Thus, for any $\lambda_i\in\mathbf{A}_A^+$, we have 
\begin{displaymath}
(\gamma^{p^s}-1)(\sum_i \lambda_i m_i)=\sum_i \gamma^{p^s}(\lambda_i)(\gamma^{p^s}-1)m_i+\sum_i (\gamma^{p^s}(\lambda_i)-\lambda_i)m_i.
\end{displaymath}
As $\gamma^{p^s}(\lambda_i)-\lambda_i\in T\mathbf{A}_A^+$ for all $i$ (recall that $\gamma$ fixes $\O_K\otimes_{\O_F}A$, and preserves $T\mathbf{A}_A^+$), it follows that $(\gamma^{p^s}-1)\fM\subseteq T\fM$, hence (4).

Finally assume that the equivalent conditions (3)--(4)--(5) hold. To show that (1) and (2) holds, it suffices to verify the conditions of \cite[Lem. D.13 (2)]{EG22}. That \cite[Lem. D.13 (2)(b)]{EG22} holds, i.e. each $g\in \langle \gamma\rangle$ acts continuously on $M$ follows by the same argument used in Lemma \eqref{gamma -1 quasi-linear}.

We now check \cite[Lem. D.13 (2)(c)]{EG22}. So let $\fM\subseteq M$ be a lattice. We need to show that $\langle \gamma^{p^s}\rangle\fM \subseteq \fM$ for all $s\geq 0$ sufficient large. As we are assuming (3), we may choose $s$ so that $(\gamma^{p^s}-1)\fM\subseteq T\fM$. In particular, $\gamma^{p^s}\fM\subseteq \fM$, and by induction we have $\gamma^{r p^s}\fM\subseteq \fM$ for all $r\geq 0$. To handle the case $r<0$, it suffices (by descending induction) to check that $\gamma^{-p^s}\fM\subseteq \fM$. This follows from the fact that $\gamma^{p^s}-1$ acts topologically nilpotent on $M$. Indeed, for any $m\in \fM$ we have 
\begin{displaymath}
\gamma^{-p^s}(m)=(1-(1-\gamma^{p^s}))^{-1}m=m+(1-\gamma^{p^s})m+(1-\gamma^{p^s})^2m+\ldots \in \fM.
\end{displaymath}
It remains to check \cite[Lem. D.13 (2)(a)]{EG22}. In other words, we need to check that for any $m\in M$, the orbit map $\langle \gamma \rangle \to M, g\mapsto gm$ is continuous at the identity of $\langle \gamma \rangle$, or equivalently, for any lattice $\fM\subseteq M$, we can find $s\geq 0$ sufficiently large so that $gm\in m+\fM$ for all $g\in \langle \gamma^{p^s}\rangle$. By picking $n$ large enough so that $m\in T^{-n}\fM$ and replacing $\fM$ by $T^{-n}\fM$, we are reduced to choosing $s$ so that $(g-1)\fM\subseteq T^n\fM$ for all $g\in \langle \gamma^{p^s}\rangle$. Again, by (3) we can choose $s$ so that $(\gamma^{p^s}-1)\fM\subseteq T^n\fM$. As in preceding paragraph, this implies in particular that $\gamma^{r p^s}\fM\subseteq \fM$ for all $r\in\mathbf{Z}$. It follows that $(\gamma^{r p^s}-1)\fM\subseteq T^n\fM$ for all $r\in \mathbf{Z}$. Indeed, if $r\geq 1$, then $(\gamma^{r p^s}-1)\fM=(\gamma^{p^s}-1)(\gamma^{(r-1)p^s}+\ldots+\gamma^{p^s}+1)\fM\subseteq (\gamma^{p^s}-1)\fM\subseteq T^n\fM$, while if $r<0$, then $(\gamma^{r p^s}-1)\fM=-(\gamma^{-r p^s}-1)\gamma^{r p^s}\fM\subseteq (\gamma^{-r p^s}-1)\fM\subseteq T^n\fM$.
\end{proof}
\end{appendices}

\medskip
\printbibliography[title ={References}]
\addcontentsline{toc}{section}{References}
\textsc{\small LAGA, Universit\' e Paris 13, 99 avenue Jean-Baptiste Cl\' ement, 93430 Villetaneuse, France}\\
\indent\textit{Email address}: \href{mailto:dat.pham@math.univ-paris13.fr}{\texttt{dat.pham@math.univ-paris13.fr}}
\end{document}